\documentclass[reqno,11pt]{amsart}
\usepackage[top=2cm]{geometry}
\usepackage{amssymb}
\usepackage{amsthm}
\usepackage{yhmath}
\usepackage{amsfonts}
\usepackage{mathtools}
\usepackage{appendix}
\usepackage[dvipsnames]{xcolor}
\usepackage{enumitem}
\usepackage[colorlinks, 
linkcolor=blue,
anchorcolor=magenta,
citecolor=blue
]{hyperref}
\usepackage{orcidlink}

\newcommand{\norm}[2]{\Vert #1 \Vert_{#2}}

\DeclareMathOperator{\diag}{diag}
\DeclareMathOperator{\Div}{div}

\DeclareMathOperator{\curl}{curl}
\usepackage[myheadings]{fullpage}

\theoremstyle{plain}
\newtheorem{theorem}{Theorem}[section]

\newtheorem{lemma}[theorem]{Lemma}
\newtheorem{proposition}[theorem]{Proposition}
\theoremstyle{definition}

\newtheorem{remark}[theorem]{Remark}

\title[Inviscid shallow water equations]{Local Well-Posedness of the Inviscid Shallow Water Equations with Surface Tension}
\date{\today}
\author[]{Xin Liu \orcidlink{0000-0001-7021-8644}}
\address[Xin Liu]{
    Department of Mathematics, Texas A\$M University, College Station, TX, 77843, USA.
}
\email{xliu23@tamu.edu}

\author[]{Bingheng Yang$^*$ \orcidlink{0009-0005-0380-1596}}
\address[Bingheng Yang]{
	Department of Mathematics, Texas A\&M University, College Station, TX, 77843, USA.
}
\email{bhyang@tamu.edu}

\numberwithin{equation}{section}

\begin{document}

\begin{abstract}
We show that the inviscid shallow water equations with gravity and surface tension 
are locally well-posed in dimensions one and two. In contrast to previous works by Benzoni-Gavage, Danchin, and Descombes \cite{BenzoniGavageDanchinDescombes1D,BenzoniGavageDanchinDescombes3D}, we work only in the real-valued phase space and make use of the symmetry of the system. Our alternative proof for local well-posedness sheds some light on the symmetric structure of the shallow water system and can be used for further study, including designing numerical scheme, dispersion estimate, etc.

\smallskip 

{\noindent\bf Keyboards:} Euler-Korteweg, shallow water equations, surface tension, local well-posedness, symmetric systems.


\end{abstract}

\maketitle

{
  \hypersetup{linkcolor=blue}
  \tableofcontents
}

\section{Introduction}
We consider the system, 
\begin{subequations}
\label{ShallowWaterEquations}
\begin{align}
\label{MaSs}
\partial_t h+\Div(hu)& = 0 \quad && \text{in} \quad \mathbb{R}^n, \\
\label{MoMeNtUm}
\partial_t u+u\cdot\nabla u+\nabla h-\nabla \Delta h & =0 \quad &&\text{in} \quad \mathbb{R}^n, \\
(h,u)\vert_{t=0} & =(h_0,u_0), 
\end{align}
\end{subequations}
in dimensions $ n = 1,2 $. We can view \eqref{ShallowWaterEquations} as a model for shallow water equations with surface tension. Here $h > 0$ denotes the height, $u \in \mathbb R^n$ denotes the velocity of the water wave, the pressure term  $\nabla h$ denotes the effect of gravity, and the third order term $-\nabla\Delta h$ arises from surface tension. Physically, \eqref{MaSs} and \eqref{MoMeNtUm} stand for the conservation of mass and momentum, respectively. The shallow water equations \eqref{ShallowWaterEquations} were formally derived from the free boundary problem for the incompressible water wave equations by the first author and his collaborators in \cite{LiLiuPeschka}, and were also previously known in the review article \cite{BreschDesjardinsMetivier} and the monograph \cite{Bresch}. As a shallow water model, the validity of \eqref{ShallowWaterEquations} was justified by Bresch and Noble in \cite{BreschNoble} in some scales. 

\smallskip
\begin{remark}
For $ n \geq 3 $, one should interpret \eqref{ShallowWaterEquations} as a special case of the Euler--Korteweg system \cite{BenzoniGavageDanchinDescombes3D, AntonelliMarcati1,BertiMasperoMurgante}, which models a fluid with quantum effects. See Section \ref{sec:irrotational-flow} for our remark about the irrotation flows in $ n $ dimension. 
\end{remark}

\subsection{Main result}
Our goal is to prove local well-posedness for \eqref{ShallowWaterEquations}. We note that the cases of $ n = 1 $ and $ n = 2 $ admit fundamentally different structures. Therefore, we will treat them separately.   

In the case of $ n = 1 $, system \eqref{ShallowWaterEquations} is reduced to the following:
\begin{subequations}
    \label{sys:1-d_sw}
    \begin{align}
        \label{1-d-Mass}
        \partial_t h + \partial_x(hu) & = 0 \quad &&\text{in} \quad  \mathbb R, 
        \\
        \label{1-d-momentum}
        \partial_t u+ u \partial_x  u + \partial_x h - \partial_x^3 h & = 0 \quad && \text{in} \quad \mathbb R,\\
        \label{1-d-initial}
        (h,u)\vert_{t=0} & = (h_0,u_0). 
    \end{align}
\end{subequations}

\begin{theorem}
    \label{thm:1-d} Consider initial data $ (h_0, u_0)$ for system \eqref{sys:1-d_sw}, with
\begin{equation}
\label{h_initial-bound}
    0 < \underline h := \inf_{x} h_{0}(x) < 1 < \overline h := \sup_x h_{0}(x) < \infty,
\end{equation}
and
\begin{equation}
    \label{1-d-initial-regularity}
    h_0 - 1 \in H^3(\mathbb R), \quad  u_0 \in H^2(\mathbb R). 
\end{equation}
Then there exists $ T \in (0,\infty) $, depending only on the initial data, such that there exists a unique solution $ (h-1,u) \in L^\infty([0,T];H^3(\mathbb R)) \times L^\infty([0,T];H^2(\mathbb R)) $ to system \eqref{sys:1-d_sw}, with 
    \begin{equation}
        \label{1-d-a-priori-estimate}
        \sup_{0\leq t\leq T} \sum_{
        0\leq j\leq 2
        }\norm{\partial_x^j(h(t)-1),\partial_x^j u(t), \partial_x^{j+1} h(t) }{L^2} < C <\infty,
    \end{equation}
    for some constant $ C \in (0,\infty) $ depending only on the initial data. 
\end{theorem}


Meanwhile, in the case of $ n=2 $, we will show
\begin{theorem}
    \label{MainTheoremLWP}
 Consider initial data $ (h_0, u_0) $ for system \eqref{ShallowWaterEquations}, with
\begin{equation}
    0 < \underline h := \inf_{x} h_{0}(x) < 1 < \overline h := \sup_x h_{0}(x) < \infty,
\end{equation}
and with regularity
\begin{equation}
    \label{n-d-initial-regularity}
    h_0 - 1 \in H^{7}(\mathbb R), \quad  u_0 \in H^{6}(\mathbb {R}^2). 
\end{equation}
Then there exists $ T \in (0,\infty) $ depending only initial data, such that there exists a unique solution $ (h-1,u) \in L^\infty([0,T];H^{4}(\mathbb R)) \times L^\infty([0,T];H^{3}(\mathbb {R}^2)) $ to system \eqref{ShallowWaterEquations}, with 
    \begin{equation}
        \sup_{0\leq t\leq T} \sum_{0\leq j\leq 3}\norm{\nabla^{j}(h(t)-1),\nabla^{j} u(t), \nabla^{j+1} h(t) }{L^2} < C <\infty,
    \end{equation}
    for some constant $ C \in (0,\infty) $ depending only on the initial data.
\end{theorem}

To prove Theorems \ref{thm:1-d} and \ref{MainTheoremLWP}, we will make use of the symmetric structure of system \eqref{ShallowWaterEquations}, which can be seen via the momentum variable $m:=hu.$ Indeed, one can write down the following equivalent system of the shallow water equations using $ (h,m) $: 
    \begin{subequations}
    \label{ShallowWaterMomentum}
    \begin{align}
    \label{ShallowWaterMomentumMass}
        \partial_t h+ \Div  m  & =0, \\
        \label{ShallowWaterMomentumMomentum}
        \frac{1}{h}\partial_t m +\frac{m}{h^2}\cdot\nabla m+ \frac{m}{h^2}\Div m+\nabla h-\nabla \Delta h & =\frac{(m\cdot\nabla)h}{h^3}  m, \\
        \label{ShallowWaterSurfaceGradient}
        \partial_t \nabla h+\nabla \Div m & =0.
    \end{align}
    \end{subequations}

\smallskip 

To explain the idea, in the case of $ n = 1 $, i.e., in one spatial dimension, let $ U:= (h-1, m, \partial_x h)^\top \in \mathbb R^3 $. One can write system \eqref{sys:1-d_sw} (or equivalently, system \eqref{ShallowWaterMomentum}) as
\begin{subequations}
\label{ShallowWaterMomentum1DMatrix}
    \begin{gather}
    \begin{pmatrix}
        1 & 0 & 0 \\
        0 & \frac{1}{h} & 0 \\
        0 & 0 & 1 \\
    \end{pmatrix}\partial_t U+\begin{pmatrix}
        0 & 1 & 0 \\
        1 & \frac{2m}{h^2} & 0 \\
        0 & 0 & 0 \\
    \end{pmatrix}\partial_x U+\begin{pmatrix}
        0 & 0 & 0 \\
        0 & 0 & -1 \\
        0 & 1 & 0 \\
    \end{pmatrix}\partial_x^2 U=\begin{pmatrix}
        0 \\
        \frac{m^2}{h^3}\partial_x h \\
        0 \\
    \end{pmatrix}, \\
    U(0,x)=(h_0-1,m_0,\partial_x h_0).
\end{gather}
\end{subequations}
One can see that system \eqref{ShallowWaterMomentum1DMatrix} resembles a symmetric hyperbolic system, with additional second order terms. In particular, the associated linear operator generates a bounded group, and therefore there is no loss of regularity; see Lemma \ref{LemmaPhi}, below. Due to this symmetry, the local well-posedness of the generalized system \eqref{ShallowWaterV} follows easily from the method of vanishing viscosity. We refer to Section \ref{vv} for the complete argument. At last, in Section \ref{Sect:WP}, we show the well-posedness of system \eqref{ShallowWaterMomentum1DMatrix} (or equivalently, system \eqref{sys:1-d_sw}) using an approximating argument.

\smallskip 

However, in dimension two, the symmetric structure of the shallow water system \eqref{ShallowWaterEquations} differs from that in \eqref{ShallowWaterMomentum1DMatrix}. Indeed, only the acoustic component of the flow admits a symmetric structure, while the vorticity component of the flow is transported. See Section \ref{sec:2d} for the detailed formulation of $ n = 2 $, as well as the \emph{a priori} estimates for the 2d system \eqref{ShallowWaterMomentum2D-1}. 
In Section \ref{Regularity-remark} we explain why the approximating argument in Section \ref{Sect:WP} cannot be directly applied to the 2d case without additional regularity assumptions.
Finally, in Section \ref{sec:irrotational-flow}, we remark about a result for irrotational flows in dimensions $\geq 3$. 


\subsection{Background and literature on inviscid water waves with gravity and surface tension} 
For inviscid shallow water equations \emph{without} surface tension, or for \emph{viscous} shallow water equations with surface tension, we refer the reader to Bresch's book \cite{Bresch}, Lannes' works \cite{LannesBook,Lannes} and the review article by Bresch--Desjardins--M{\'e}tivier \cite{BreschDesjardinsMetivier} for an extensive overview on the mathematical properties of shallow water models in the absence/presence of boundary. In particular, Bresch and Noble in \cite{BreschNoble} rigorously justified the validity of \eqref{ShallowWaterEquations} in smooth function spaces.


Much more mathematical studies on inviscid water waves are available without the assumption of shallowness. Here we highlight the following works: Wu \cite{WUI,WUII,WUIII,WUIV}, Coutand--Shkoller \cite{CoutandShkoller}, Alazard--Burq--Zuily \cite{AlazardBurqZuily1,AlazardBurqZuily2,AlazardBurqZuily3} and Deng--Ionescu--Pausader--Pusateri \cite{IonescuPusateriI,IonescuPusateriII,IonescuPusateriIII,DengIonescuPausaderPusateri} managed to show local and global well-posedness of incompressible irrotational water waves with or without surface tension using various methods, whereas Castro--C{\'o}rdoba--Fefferman--Gancedo--G{\'o}mez-Serrano \cite{Castro1,Castro2} investigated the finite-time formation of splash singularities. Finally, we note that Ambrose--Masmoudi \cite{AmbroseMasmoudi1,AmbroseMasmoudi2} and Agrawal \cite{Agrawal} investigated the zero surface tension limit problem. 


    We remark here that Theorem \ref{thm:1-d} and Theorem \ref{MainTheoremLWP} corresponds to the $K=1$ case for the Euler--Korteweg system in \cite[Theorem 1.1]{BenzoniGavageDanchinDescombes3D}. The proof in \cite{BenzoniGavageDanchinDescombes3D} relies on finding the good variable 
    \[
    w:=\frac{\nabla h}{\sqrt{h}},
    \]
    then complexifying the equations and performing rather involved weighted energy estimates. Moreover, the argument in \cite{BenzoniGavageDanchinDescombes3D} requires a vanishing viscosity argument using the hyperdissipative term $-\varepsilon \Delta^2$, for which uniform-in-$\varepsilon$ estimates are more difficult to establish. We provide a shorter and more physically intuitive proof based on another good variable \eqref{ThetaVariable}, which in turn gives more insights to the hidden symmetric structure of \eqref{ShallowWaterEquations}. In addition, we believe that our approach based on (almost) symmetric systems can be useful for further numerical studies of shallow water waves. See for instance \cite{BreschEtAl,EiterGiesselmannLasarzikOffnerSauerborn,Guermond1, Guermond2,Guermond3,NobleVila,Xing,XingShu,XingZhangShu} for numerical studies of shallow water models and of the Euler--Korteweg system.

\subsection{Notations}
\label{Notations}
In this paper, we denote $A\lesssim B $ if
\[
A\leq CB,
\]
for some generic constant $C \in (0,\infty)$ that may be different from line to line. We will use subscripts for $ C $ to emphasize the dependency. \\

In addition, for $f_1,\cdots,f_n\in X$, we denote
\[
\Vert f_1,\cdots,f_n\Vert_{X}^2:=\sum_{j=1}^n\Vert f_j\Vert_{X}^2.
\]
Here $ X $ is some Banach space. We denote
\[
\int (\cdot)d\vec{x}:=\int_{\mathbb R^n}(\cdot) d\vec{x}=\int_{\mathbb R^n} (\cdot) dx_1\ldots dx_n,
\]
for integrals on $\mathbb R^n$, and the pairing
\[
(\cdot,\cdot)
\]
stands for the standard $L^2$ inner product. We will use
\begin{equation}
    \label{partial-z}
    \partial \in \{\partial_{x_1},\partial_{x_2},\cdots,\partial_{x_n}\},
\end{equation}
to denote a spatial derivative along arbitrary direction. Lastly, the notation
\[
(\cdot)^\top
\]
stands for the transpose of a matrix.

\section{\textit{A priori} estimates in 1D}
\label{1DEnergyEstimates}
In this section, we focus on the {\it a priori} energy estimates for the 1d shallow water system. We write \eqref{ShallowWaterMomentum1DMatrix} as
\begin{equation}
\label{ShallowWaterMomentumConsice}
A_0(U)\partial_t U+A_1(U)\partial_x U+A_2\partial_x^2 U=S(U),
\end{equation}
where
\begin{equation}
    \label{def:U}
    U := \begin{pmatrix}
        h - 1 \\ m \\ \partial_x h
    \end{pmatrix},
\end{equation}
and
\begin{equation}\label{def:As}
\begin{gathered}
    A_0(U) := \begin{pmatrix}
        1 & 0 & 0 \\
        0 & \frac{1}{h} & 0 \\
        0 & 0 & 1 \\
    \end{pmatrix}, \quad A_1(U):=\begin{pmatrix}
        0 & 1 & 0 \\
        1 & \frac{2m}{h^2} & 0 \\
        0 & 0 & 0 \\
    \end{pmatrix}, \quad A_2:=\begin{pmatrix}
        0 & 0 & 0 \\
        0 & 0 & -1 \\
        0 & 1 & 0 \\
    \end{pmatrix},  \\
    S(U):=\begin{pmatrix}
        0 \\
        \frac{m^2}{h^3}\partial_x h \\
        0 \\
    \end{pmatrix}. 
    \end{gathered}
\end{equation} 
We will make use of the symmetric structure of equation \eqref{ShallowWaterMomentumConsice}.

\subsection{\textit{A priori} assumption,  the generalized system, and the initial data}
\label{subsec:generalized-system-initial}

Instead of \eqref{ShallowWaterMomentumConsice},
it is more convenient to consider a generalized system.
Let
\begin{equation}
\label{Def:V}
\begin{gathered}
V:=\begin{pmatrix}
    \eta -1 \\
    \omega \\
    \phi
\end{pmatrix}, \quad A_0=A_0(V):=\begin{pmatrix}
    1 & 0 & 0 \\
    0 & \frac{1}{\eta} & 0 \\
    0 & 0 & 1
\end{pmatrix}, \quad A_1=A_1(V):=\begin{pmatrix}
    0 & 1 & 0 \\
    1 & \frac{2\omega}{\eta^2} & 0 \\
    0 & 0 & 0
\end{pmatrix},  \\
{S}_{\phi} = S_\phi(V):=\begin{pmatrix}
    0 \\
    \frac{m^2}{\eta^3}\phi \\
    0 \\
\end{pmatrix},
\end{gathered}
\end{equation}
and the generalized system
\begin{equation}
    \label{ShallowWaterV}
    A_0\partial_t V+A_1\partial_x V+A_2\partial_x^2 V={S}_{\phi}.
\end{equation}
Notice that, as long as the solutions to \eqref{ShallowWaterMomentumConsice} and \eqref{ShallowWaterV} exist with the same initial data and are sufficient regular, it is easy to prove $ U(x,t) = V(x,t) $ for all $ 0 \leq t \leq T $ and all $x\in\mathbb{R}$, where $ T $ is the maximal existing time.  

For the generalized system \eqref{ShallowWaterV}, consider the energy functional
\begin{equation}
    \label{Energy}
    \mathcal E(t):= \sum_{\substack{i+j\leq 2,\\ i,j \geq 0}} \Vert\partial_t^i\partial_x ^j (\eta-1),{\partial_t^i\partial_x^j \omega},\partial_t^i\partial_x^{j}\phi\Vert_{L^2}^2,
\end{equation}
and the intermediate energy
\begin{equation}
    \label{intrmd-energy}
    \mathfrak E(t): = \sum_{\substack{i+j\leq 2,\\ i,j \geq 0}} (A_0(V) \partial_t^i \partial_x^j V (t), \partial_t^i \partial_x^j V(t)). 
\end{equation}
In particular, provided that 
\begin{equation}
    \label{non-degenerate}
    0<\frac{1}{2}\underline h \leq \inf_{x,t} \eta(x,t) \leq \sup_{x,t} \eta(x,t) \leq 2 \overline h, 
\end{equation}
there exists a constant $ C_{\underline h, \overline h} \in (0,\infty) $ depending on $ \underline h, \overline h $, such that
\begin{equation}
    \label{energy-eqvlt}
    \frac{1}{C_{\underline h, \overline h}}\mathcal E(t) \leq \mathfrak E(t) \leq C_{\underline h, \overline h} \mathcal E(t).
\end{equation}
Moreover, at $ t = 0 $, we choose initial data with regularity
\begin{equation}\label{V-initialdata} (\eta-1,\omega, \phi)\vert_{t=0} := (\eta_0 -1, \omega_0, \phi_0 ) \in H^3(\mathbb{R}) \times H^4(\mathbb{R}) \times H^4(\mathbb{R}), \end{equation} 
as the initial data for $ \partial_t \eta, \partial_t \omega, \partial_t \phi, \partial_{tt} \eta, \partial_{tt} \omega, \partial_{tt} \phi $ are explicitly given by
\begin{equation}
\label{initial_data_time}
\begin{aligned}
    \partial_t \eta(t=0) & = - \partial_x \omega_0,\\
    \partial_t \omega(t=0) &= - \frac{2\omega_0\partial_x \omega_0}{\eta_0}  - \eta_0 \partial_x \eta_0 + \eta_0 \partial_{x}^2 \phi_0 + \frac{\omega_0^2 \partial_x \eta_0}{\eta_0^2}, \\ 
    \partial_t \phi (t=0) & = - \partial_x^2 \omega_0, \\
    \partial_t^2 \eta(t=0) &= -2\frac{\omega_0 }{\eta_0^2} \phi_0\partial_x \omega_0-\frac{\omega_0^2}{\eta_0^2 }\partial_x \phi_0+\frac{2\omega_0^2}{\eta_0^3}\phi_0\partial_x\eta_0-\partial_x\eta_0 \partial_x^2\phi_0-\eta_0\partial_x^3 \phi_0+(\partial_x\eta_0)^2  \\
    & +\eta_0\partial_x^2 \eta_0+\frac{2(\partial_x\omega_0)^2}{\eta_0}+\frac{2\omega_0}{\eta_0}\partial_x^2 \omega_0-\frac{2\omega_0}{\eta_0^2 } \partial_x \eta_0\partial_x \omega_0, \\
     \partial_t^2 \phi (t=0) &= \partial_x \partial_t^2 \eta(t=0) \\
     & = -\frac{2\phi_0}{\eta_0^2 } (\partial_x \omega_0)^2 -\frac{4\omega_0}{\eta_0^2 } \partial_x \phi_0 \partial_x \omega_0-2\frac{\omega_0}{\eta_0^2 } \phi_0\partial_x^2 \omega_0+4\frac{\omega_0}{\eta_0^3} \phi_0\partial_x \eta_0 \partial_x \omega_0 -\frac{\omega_0^2}{\eta_0^2}\partial_x^2 \phi_0\\
     &+\frac{4\omega_0^2}{\eta_0^3} \partial_x \eta_0\partial_x \phi_0+\frac{4\omega_0}{\eta_0^3}\phi_0\partial_x \omega_0\partial_x \eta_0+\frac{2\omega_0^2}{\eta_0^3}\phi \partial_x^2 \eta_0 -\frac{6\omega_0^2}{\eta_0^4}\phi_0 (\partial_x \eta_0)^2-\partial_x^2\eta_0 \partial_x^2 \phi_0\\
     & -2\partial_x \eta_0 \partial_x^3 \phi_0-\eta_0\partial_x^4 \phi_0+3\partial_x\eta_0\partial_x^2\eta_0+\eta_0\partial_x^3 \eta_0+\frac{6}{\eta_0}\partial_x\omega_0\partial_x^2\omega_0-\frac{2(\partial_x\omega_0)^2}{\eta_0^2}\partial_x\eta_0+\frac{2\omega_0}{\eta_0}\partial_x^3\omega_0 \\
     & -\frac{2\omega_0}{\eta_0^2}\partial_x^2 \omega_0-\frac{2}{\eta_0^2}\partial_x\eta_0 (\partial_x\omega_0)^2-\frac{2\omega_0}{\eta_0^2}\partial_x^2\eta_0\partial_x \omega_0-\frac{2\omega_0}{\eta_0^2}\partial_x\eta_0\partial_x^2 \omega_0+\frac{4\omega_0}{\eta_0^3}(\partial_x \eta_0)^2 \partial_x \omega_0,\\
     \partial_t^2 \omega (t=0) &=-\frac{2\partial_x \omega_0}{\eta_0}\left(- \frac{2\omega_0\partial_x \omega_0}{\eta_0}  - \eta_0 \partial_x \eta_0 + \eta_0 \partial_{x}^2 \phi_0 + \frac{\omega_0^2 \partial_x \eta_0}{\eta_0^2}\right)\\
     &-\frac{2\omega_0}{\eta_0}\partial_x \left(- \frac{2\omega_0\partial_x \omega_0}{\eta_0}  - \eta_0 \partial_x \eta_0 + \eta_0 \partial_{x}^2 \phi_0 + \frac{\omega_0^2 \partial_x \eta_0}{\eta_0^2}\right)\\
     &-\frac{2\omega_0\partial_x \omega_0}{\eta_0^2}\partial_x\omega_0+\partial_x\eta_0\partial_x \omega_0 -\partial_x\omega_0\partial_x^2\phi_0-\eta_0\partial_x^4 \omega_0  \\
     & +\frac{2\omega_0\partial_x\eta_0}{\eta_0^2} \left(- \frac{2\omega_0\partial_x \omega_0}{\eta_0}  - \eta_0 \partial_x \eta_0 + \eta_0 \partial_{x}^2 \phi_0 + \frac{\omega_0^2 \partial_x \eta_0}{\eta_0^2}\right)-\frac{\omega_0^2}{\eta_0^2} \partial_x^2 \omega_0+\frac{2\omega_0^2\partial_x\eta}{\eta_0^3} \partial_x \omega_0.
\end{aligned}
\end{equation}
It follows from \eqref{initial_data_time} that 
$ \mathcal E(0) < \infty $ if \eqref{non-degenerate} and \eqref{V-initialdata} hold.

\subsection{\textit{A priori} estimates of the generalized system \eqref{ShallowWaterV}} 
We focus first on obtaining the {\it a priori} energy estimates for the solution $ V $ to system \eqref{ShallowWaterV}. We assume the non-degenerate condition \eqref{non-degenerate} holds
for all $0\leq t \leq T$, where $T$ is the maximum existence time. We will justify \eqref{non-degenerate} in Section \ref{Justification} below.

Our goal of this section is to establish the following proposition:
\begin{proposition}
    \label{prop:a-priori-V}
    Under the assumption \eqref{non-degenerate}, the solution to \eqref{ShallowWaterV} satisfies
    \begin{equation}
        \label{est:a-priori-energy-0}
        \sup_{0\leq t \leq T} \mathfrak E(t) \leq 2 \mathfrak  E(0),
    \end{equation}
    and thus, thanks to \eqref{energy-eqvlt}, 
    \begin{equation}
        \label{est:a-priori-energy}
        \sup_{0\leq t \leq T} \mathcal E(t) \leq C_{\underline h, \overline h} \mathcal  E(0),
    \end{equation}
    for some $ T \in (0,\infty) $ and some $ C_{\underline h, \overline h} \in (0,\infty) $ depending only on the initial data and $ \underline h, \overline h $. 
\end{proposition}

The proof of Proposition \ref{prop:a-priori-V} relies on the following property of the linear operator $ A_0 \partial_t + A_1 \partial_x + A_2 \partial_x^2$:
\begin{lemma}
\label{LemmaPhi}
    Let $\Phi: \mathbb R \times \mathbb R \to \mathbb R^3$ be the solution to the linear equation
    \begin{equation}
    \label{PhiEq}
    A_0\partial_t \Phi+A_1\partial_x \Phi+A_2\partial_x^2 \Phi=G,
    \end{equation}
    where $G\in L_t^{\infty}L^2_x$. Then for a.e. $ t $, we have
    \begin{equation}
    \label{LemmaPhiEstimate}
        \frac{d}{dt} (A_0 \Phi,\Phi) \leq C_{\underline h, \overline h} (\mathcal E(t) + 1) \Vert\Phi\Vert_{L^2}^2+ 2 \Vert G\Vert_{L^2} \Vert\Phi\Vert_{L^2}.
    \end{equation}
\end{lemma}
\begin{proof}
    Taking the $ L^2 $-inner product of \eqref{PhiEq} with $ 2 \Phi $ and integrating the resultant by part lead to 
    \begin{equation}
        \label{phi:001}
        \dfrac{d}{dt} (A_0 \Phi, \Phi) =  (\partial_t A_0 \Phi, \Phi) + (\partial_x A_1 \Phi, \Phi) + \underbrace{2(A_2 \partial_x \Phi, \partial_x \Phi)}_{=0} + 2 (G,\Phi),
    \end{equation}
    where we used the fact that $ A_0, A_1 $ are symmetric matrices and $ A_2 $ is a constant skew-symmetric matrix. Applying the Hölder inequality to the right-hand side of \eqref{phi:001} yields
    \begin{equation}
    \label{phi:002}
        \dfrac{d}{dt}(A_0 \Phi, \Phi) \leq \norm{\partial_t A_0, \partial_x A_1}{L^\infty}\norm{\Phi}{L^2}^2 + 2 \norm{G}{L^2} \norm{\Phi}{L^2}. 
    \end{equation}
    It remains to estimate $ \norm{\partial_t A_0, \partial_x A_1}{L^\infty} $. From \eqref{def:As} and \eqref{ShallowWaterV}, one has
    \begin{equation}
        \label{phi:003}
        \begin{gathered}
        \norm{\partial_t A_0, \partial_x A_1}{L^\infty} \leq C_{\underline h} \norm{\partial_t \eta, \partial_x \omega, \omega \partial_x \eta}{L^\infty} \leq C_{\underline h} (\norm{\partial_x \omega, \omega,\partial_x \eta}{L^\infty}^2 + 1) \\
        \leq C_{\underline h} (\sum_{j=0}^2 \norm{\partial_x^j \omega, \partial_x^j (\eta-1)}{L^2}^2 + 1) \leq C_{\underline h, \overline h} (\mathcal E(t) + 1),
        \end{gathered}
    \end{equation}
    where we have used the Sobolev embedding inequality $ \norm{\cdot}{L^\infty} \lesssim \norm{\cdot }{H^1} $. 
    This finishes the proof. 
\end{proof}
Now we are ready to establish the {\it a priori} estimates in the generalized system \eqref{ShallowWaterV}. Applying \eqref{LemmaPhiEstimate} to \eqref{ShallowWaterV} immediately leads to 
\begin{equation}
    \label{basic:energy}
    \dfrac{d}{dt}(A_0V,V) \leq C_{\underline h,\overline h} (\mathcal E(t) + 1) \norm{V}{L^2}^2 + 2 \norm{S_\phi}{L^2} \norm{V}{L^2}.
\end{equation}
Meanwhile, one has the following systems for $ V_{tt} := \partial_{t}^2 V, V_{xt} := \partial_x\partial_{t} V $, and $ V_{xx} := \partial_{x}^2 V $:
\begin{equation}
    \label{sys:dd-V}
    A_0 \partial_t V_{a} + A_1 \partial_x V_{a} + A_2 \partial_{x}^2  V_a = G_a, \qquad \text{for} ~ a \in \lbrace tt, xt, xx\rbrace,
\end{equation}
where
\begin{align}
    \label{id:tt-G}
    & \begin{aligned}
     & G_{tt} := \partial_{t}^2 S_\phi - 2 \partial_t A_0 \partial_{t}^2 V - \partial_{t}^2 A_0 \partial_t V - 2 \partial_t A_1 \partial_t\partial_{x}V- \partial_{t}^2 A_1 \partial_x V,
    \end{aligned}
    \\
    & \begin{aligned}
        & G_{xt} := \partial_{xt} S_\phi - \partial_x A_0 \partial_{tt} V - \partial_t A_0 \partial_{xt} V - \partial_{xt} A_0 \partial_t V \\
        & \qquad - \partial_x A_1 \partial_{xt}V - \partial_{t} A_1 \partial_{xx} V - \partial_{xt} A_1 \partial_x V, 
    \end{aligned}
    \label{id:xt_G}
    \\
    \label{id:xx_G}
    & \begin{aligned}
        & G_{xx} := \partial_{xx} S_\phi - 2 \partial_x A_0 \partial_{xt} V - \partial_{xx}A_0 \partial_t V - 2 \partial_x A_1 \partial_{xx} V - \partial_{xx} A_1 \partial_x V. 
    \end{aligned}
\end{align}
Then applying \eqref{LemmaPhiEstimate} in Lemma \ref{LemmaPhi} to \eqref{sys:dd-V}, one has that
\begin{equation}
    \label{h2:001}
    \dfrac{d}{dt} \sum_{a \in \lbrace tt, xt, xx\rbrace} (A_0 V_a, V_a) \leq C_{\underline h, \overline{\eta}} (\mathcal E(t) + 1) \norm{V_{tt}, V_{xt}, V_{xx}}{L^2}^2 + 2\norm{G_{tt},G_{xt},G_{xx}}{L^2} \norm{V_{tt}, V_{xt}, V_{xx}}{L^2}. 
\end{equation}
Moreover, one can directly calculate that 
\begin{equation}
\label{h-1-eq}
\begin{gathered}
    \dfrac{d}{dt}\lbrace (A_0 \partial_x V, \partial_x V) + (A_0 \partial_t V, \partial_t V) \rbrace = 2 (A_0 \partial_{x}\partial_t V, \partial_x V) + 2 (A_0 \partial_{t}^2 V, \partial_t V) \\ + (\partial_t A_0 \partial_x V, \partial_x V) + (\partial_t A_0 \partial_t V,\partial_t V)\\
    \lesssim \norm{A_0,\partial_t A_0}{L^\infty} \norm{\partial_x V, \partial_{x}\partial_tV,\partial_{t}V,\partial_{t}^2 V}{L^2}^2 \\
    \overset{\eqref{phi:003}}{\leq} C_{\underline h, \overline h} (\mathcal E(t) + 1)\norm{\partial_x V, \partial_{x}\partial_tV,\partial_{t}V,\partial_{t}^2 V}{L^2}^2.
    \end{gathered}
\end{equation}
Combining \eqref{basic:energy}, \eqref{h2:001}, and \eqref{h-1-eq} leads to
\begin{equation}
    \label{h-002}
    \dfrac{d}{dt}\mathfrak E(t)\leq C_{\underline h, \overline h} (\mathcal E(t) + 1)^2 + 2 \norm{S_\phi,G_{tt},G_{xt},G_{xx}}{L^2} (\mathcal E(t))^{\frac{1}{2}}.
\end{equation}
It remains to estimate $ \norm{S_\phi,G_{tt},G_{xt},G_{xx}}{L^2} $. However, this is immediate. Since $ A_0, A_1, S_\phi $ depend only on $ V $ but not the derivative(s) of $ V $, one can write that, for $ \Psi \in \lbrace A_0, A_1, S_\phi \rbrace $, 
\begin{align}
    \label{est:nonlinearity-001}
    & \norm{\partial_x \Psi}{L^\infty} \leq (\norm{\Psi}{L^{\infty}}+\norm{ \partial_x \Psi }{L^\infty}) \norm{\partial_x V}{L^\infty} \leq C_{\underline h,\overline h} \mathcal P(\norm{V}{L^\infty}) \norm{\partial_x V}{L^\infty}, \\
    & \norm{\partial_t \Psi}{L^\infty} \leq (\norm{\Psi}{L^{\infty}}+\norm{ \partial_x \Psi }{L^\infty}) \norm{\partial_t V}{L^\infty} \leq C_{\underline h,\overline h} \mathcal P(\norm{V}{L^\infty}) \norm{\partial_t V}{L^\infty}, \\
    & \label{est:nonlinearity-003} \begin{aligned}& \norm{\partial_{x}\partial_t \Psi}{L^2} \leq \norm{\partial_x \Psi}{L^\infty} \norm{\partial_{x}\partial_t V}{L^2} + \norm{\partial_{x}\partial_t\Psi}{L^\infty} \norm{\partial_x V}{L^\infty}\norm{\partial_t V}{L^2} \\
    & \qquad \leq C_{\underline h, \overline h} \mathcal P(\norm{V}{L^\infty}) (\norm{\partial_{x}^2 V}{L^2} + \norm{\partial_x V}{L^\infty}\norm{\partial_t V}{L^2} ),
    \end{aligned}\\
    & \label{est:nonlinearity-004} \begin{aligned}& \norm{\partial_{t}^2 \Psi}{L^2} \leq \norm{\partial_x \Psi}{L^\infty} \norm{\partial_{t}^2 V}{L^2} + \norm{\partial_{x}^2\Psi}{L^\infty} \norm{\partial_t V}{L^\infty}\norm{\partial_t V}{L^2} \\
    & \qquad \leq C_{\underline h, \overline h} \mathcal P(\norm{V}{L^\infty}) (\norm{\partial_{t}^2 V}{L^2} + \norm{\partial_t V}{L^\infty}\norm{\partial_t V}{L^2} ),
    \end{aligned}
\end{align}
where, hereafter, $\mathcal{P}(\cdot) $ is a generic non-decreasing non-negative polynomial of the argument, different from line to line. Therefore, after applying the Sobolev embedding inequality, we arrive at
\begin{equation}
    \label{est:nonlinear-005}
    \sum_{\Psi\in\lbrace A_0, A_1, S_\phi \rbrace}\lbrace \norm{\partial_x \Psi, \partial_t \Psi}{L^\infty} + \norm{ \partial_{x}\partial_t \Psi,\partial_{t}^2 \Psi}{L^2} \rbrace \leq  \mathcal P(\mathcal E(t)^{\frac{1}{2}}).
\end{equation}
Now we conclude with
\begin{equation}
    \label{est:nonlinear-006}
    \begin{gathered}
    \norm{S_\phi,G_{tt},G_{xt},G_{xx}}{L^2}  \leq \sum_{\Psi \in \lbrace A_0, A_1, S_\phi\rbrace} \bigl\lbrace \norm{\partial_x \Psi, \partial_t \Psi}{L^\infty} \norm{\partial_{x}\partial_tV,\partial_{t}^2V,\partial_{x}^2V}{L^2} \\
    +  \norm{\partial_{x}\partial_t \Psi, \partial_{t}^2 \Psi}{L^2} \norm{\partial_{x}V,\partial_{t}V}{L^\infty} \rbrace + \norm{S_\phi, \partial_{x}^2S_\phi, \partial_{x}\partial_t S_\phi, \partial_{t}^2 S_\phi}{L^2} 
    \leq  \mathcal P(\mathcal E(t)^{\frac{1}{2}}).
    \end{gathered}
\end{equation}
Hence, from \eqref{h-002} and \eqref{est:nonlinear-006}, thanks to \eqref{energy-eqvlt}, there exists $T>0$ such that
\begin{equation}
    \label{energy-ineq-total}
    \dfrac{d}{dt}\mathfrak E(t) \leq  \mathcal P(\mathfrak E(t)^{\frac{1}{2}}),
\end{equation}
for all $t\in (0,T)$. This finishes the proof of Proposition \ref{prop:a-priori-V} after integrating \eqref{energy-ineq-total}.

\subsection{Justifying the \textit{a priori} assumption \eqref{non-degenerate}}
\label{Justification}
We now turn to the assumption \eqref{non-degenerate}. Our goal is to show the following proposition:
\begin{proposition}
    \label{prop:a-priori-assumption}
    Suppose that 
    \begin{equation}
    \label{est:energy}
        \sup_{0\leq t \leq T} \mathcal E(t) < C_{\underline h, \overline h} \mathcal E(0),
    \end{equation}
    for some $ T \in(0,\infty) $ and $ C_{\underline h, \overline h} \in (0,\infty) $. There exists $ T'\in (0,T] $ such that 
    \begin{equation}
    \label{est:height}
       0<\frac{3}{4} \underline h \leq \inf_{x,t\in [0,T']} \eta(x,t) \leq \sup_{x,t\in [0,T']} \eta(x,t) \leq \frac{3}{2} \overline h.
    \end{equation}
    In particular, \eqref{est:height} is a stronger estimate than \eqref{non-degenerate}.
    Therefore, after updating $ T $ with $ T' $, both \eqref{est:energy} and \eqref{est:height} hold in the interval $ [0,T] $. 
\end{proposition}

\begin{proof}
    Given $ t \in [0,T] $,
    \begin{equation}
    \label{est:h-001}
        \norm{\partial_t \eta(t)}{L^\infty} \leq \norm{\partial_t \eta(t),\partial_{x}\partial_t \eta(t)}{L^2} \lesssim \mathcal E(t)^{\frac{1}{2}} \lesssim \mathcal E(0)^{\frac{1}{2}},
    \end{equation}
    by the Sobolev embedding inequality. One has that, for all $ t \in [0,T] $,
    \begin{equation}
    \label{est:h-002}
    \begin{gathered}
    \eta(x,t)  = \eta_0(x) + \int_0^t \partial_t \eta(x,\tau) d\tau \leq \eta_0(x) + \int_0^t \norm{\partial_t \eta(\tau)}{L^\infty} \,d\tau\\
    \leq \overline h + t C \mathcal E(0)^{\frac{1}{2}},
    \end{gathered}
    \end{equation}
    for some $ C \in (0,\infty) $. 
    Therefore, by choosing $ t $ small enough, \eqref{est:h-002} implies the last inequality in \eqref{est:height}. The other inequality in \eqref{est:height} follows similarly. 
\end{proof}

We henceforth conclude from Propositions \ref{prop:a-priori-V} and \ref{prop:a-priori-assumption} via boostrap arguments with the following:
\begin{proposition}
    \label{EnergyProposition}
    Consider $ (\eta_0 -1, \omega_0, \phi_0 ) \in H^3(\mathbb{R}) \times H^4(\mathbb{R}) \times H^4(\mathbb{R}) $, i.e., initial data as in \eqref{V-initialdata}. 
    There exists some time $T>0$ such that if $(\eta-1,\omega,\phi)$ solves \eqref{ShallowWaterV}, then
    \begin{equation}
    \label{est:energy-total}
        \sup_{0\leq t\leq T} \sum_{\substack{i+j\leq 2,\\ i,j \geq 0}} \Vert\partial_t^i\partial_x ^j (\eta-1),\partial_t^i\partial_x^j \omega,\partial_t^i\partial_x^{j}\phi\Vert_{L^2}^2 \leq C_{\underline h, \overline h} \mathcal E(0),
    \end{equation}
    for some constant $ C_{\underline h, \overline h} \in (0,\infty) $ depending only on $ \underline h, \overline h $. 
\end{proposition}

\section{Local well-posedness of the 1d generalized system \eqref{ShallowWaterV} via vanishing viscosity}
\label{vv}

\subsection{Vanishing viscosity and the approximating system}

In this section, we will focus on the construction of the local-in-time solution and establish the well-posedness of system \eqref{ShallowWaterV}. Notice that the system \eqref{ShallowWaterV} becomes singular in $ \lbrace h = 0 \rbrace $. To construct a solution to the nonlinear problem, one will need to avoid this possible degeneracy. 
For $ \varepsilon \in (0,1) $, let 
\begin{equation}
\label{def:jpn-bracket}
\langle (\cdot)\rangle_{\varepsilon}:=\sqrt{(\cdot)^2+\varepsilon^2}.
\end{equation}
and 
\begin{equation}\label{vsc-As}\begin{gathered}
A_0^{\varepsilon}:=\begin{pmatrix}
1 & 0 & 0 \\
0 & \langle \eta^\varepsilon\rangle_{\varepsilon}^{-1} & 0 \\
0 & 0 & 0\\
\end{pmatrix}, \quad 
A_1^{\varepsilon}:=
\begin{pmatrix}
    0 & 1 & 0 \\
    1 & 2\omega^\varepsilon\langle \eta^\varepsilon\rangle_{\varepsilon}^{-2} & 0 \\
    0 & 0 & 0 \\
\end{pmatrix},
\\
A_2^{\varepsilon}:=\begin{pmatrix}
        -\varepsilon & 0 & 0 \\
        0 & -\varepsilon & -1 \\
        0 & 1 & -\varepsilon \\
    \end{pmatrix},
    \quad
    S_{\phi}^{\varepsilon}:=
    \begin{pmatrix}
        0 & \\
        (\omega^\varepsilon)^2 \langle \eta^\varepsilon\rangle_{\varepsilon}^{-3}\phi^\varepsilon\\
        0
    \end{pmatrix}.
\end{gathered}
\end{equation}
Then we consider the following regularization of system \eqref{ShallowWaterV}:
\begin{subequations}
\label{ShallowWaterMomentumgeneralized}
\begin{gather}
\label{ShallowWaterMomentumgeneralizedConsice}
A_0^\varepsilon \partial_t V^\varepsilon + A_1^\varepsilon \partial_x V^\varepsilon + A_2^\varepsilon \partial_{x}^2 V^\varepsilon = S_\phi^\varepsilon, \\
    V^\varepsilon(0,x)=V_0^\varepsilon = (\eta_0^{\varepsilon}-1, \omega_0^{\varepsilon}, \phi_0^{\varepsilon}), 
\end{gather}
\end{subequations}
where the initial data satisfies 
\begin{equation}
    \label{Vsc_initial}
    \begin{aligned}
    \eta_0^\varepsilon - 1 & \xrightarrow{\varepsilon\to 0} \eta_0-1  \qquad  & \text{in} ~ H^3(\mathbb R), \\
    (\omega_0^{\varepsilon}, ~ \phi_0^{\varepsilon}) & \xrightarrow{\varepsilon\to 0} (\omega_0, ~ \phi_0) & \text{in} ~ H^4(\mathbb R),
    \end{aligned}
\end{equation}
and 
\begin{align}
    \label{Vsc_initial-2}
    \varepsilon (\eta_0^\varepsilon-1) & \xrightarrow{\varepsilon\to 0} 0  \qquad   \text{in} ~ H^4(\mathbb R). 
\end{align}
The energy associated to \eqref{ShallowWaterMomentumgeneralized} is given by
\begin{equation}
    \label{EpsilonEnergy}
    \mathfrak E^{\varepsilon}(t):=\sum_{\substack{i+j\leq 2,\\ i,j \geq 0}} \Vert\partial_t^i\partial_x^j (\eta^{\varepsilon}-1),\frac{\partial_t^i\partial_x^j \omega^{\varepsilon}}{{\langle \eta^{\varepsilon}\rangle}^{\frac{1}{2}}_{\varepsilon}},\partial_t^i\partial_x^{j}\phi^{\varepsilon}\Vert_{L^2}^2.
\end{equation}
For any fixed $ \varepsilon $, system \eqref{ShallowWaterMomentumgeneralized} is a non-degenerate uniformly parabolic system, for which there exists a local-in-time solution through construction in, for instance, \cite{Taylor}. Note that the existence time of the said solution may depend on $ \varepsilon $. However, one can immediately check that the {\it a priori} estimates in Section \ref{1DEnergyEstimates} still hold for $ V^\varepsilon $, uniformly in $ \varepsilon $. Therefore, one can conclude the following proposition:
\begin{proposition}
    \label{WellPosedViscousProblem}
    Given initial data $ V^\varepsilon_0 = (\eta_0^{\varepsilon}-1,\omega_0^{\varepsilon}, \phi_0^{\varepsilon})\in H^4(\mathbb{R})\times H^4(\mathbb{R})\times H^4(\mathbb{R})$, 
    there exists $ \mathfrak{T}\in (0,\infty) $, independent of $ \varepsilon $, such that there exists a unique solution $ V^\varepsilon = (\eta^\varepsilon-1,\omega^\varepsilon,\phi^\varepsilon) \in  L^\infty(0,\mathfrak{T};H^2(\mathbb{R})) $ and $ \partial_t V^\varepsilon \in L^\infty(0,\mathfrak{T};H^1(\mathbb R)) $, satisfying
    \begin{equation}
        \label{energy-viscous}
        \sup_{0\leq t \leq \mathfrak{T}} \sum_{\substack{i + j \leq 2, \\
        i \geq 0, j \geq 0}} \norm{\partial_t^i\partial_x^j V^\varepsilon(t)}{L^2} \leq C (\varepsilon\norm{ \eta^\varepsilon_0-1}{H^4}+\norm{\eta_0^\varepsilon-1}{H^3} + \norm{\omega^\varepsilon_0,\phi^\varepsilon_0}{H^{4}} ),
    \end{equation}
    for some constant $ C \in (0,\infty) $, depending only on $ \underline h, \overline h $ but independent of $ \varepsilon \in (0,1) $. 
\end{proposition}

\begin{proof}
From Proposition \ref{EnergyProposition}, we only need to verify the initial data. In particular, thanks to similar expansion of the initial data as in \eqref{initial_data_time}, the only item we need to verify is $\varepsilon \norm{ \eta_0^{\varepsilon}-1}{H^4}$. However, notice that for the viscous system \eqref{ShallowWaterMomentumgeneralizedConsice}, we have
\[
\partial_t \eta^{\varepsilon}(t=0)=\varepsilon \partial_x^2 \eta^{\varepsilon}_0-\partial_x\omega^{\varepsilon}_0,
\]
and
\begin{align*}
    \partial_t^2 \eta^{\varepsilon}(t=0) &= -2\frac{\omega_0^{\varepsilon} }{(\eta_0^{\varepsilon})^2} \phi_0^{\varepsilon}\partial_x \omega_0^{\varepsilon}-\frac{(\omega_0^{\varepsilon})^2}{(\eta_0^{\varepsilon})^2 }\partial_x \phi_0^{\varepsilon}+\frac{2(\omega_0^{\varepsilon})^2}{(\eta_0^{\varepsilon})^3}\phi^{\varepsilon}_0\partial_x\eta^{\varepsilon}_0-\partial_x\eta_0^{\varepsilon} \partial_x^3\phi_0^{\varepsilon}-\eta_0^{\varepsilon}\partial_x^3 \phi_0^\varepsilon+(\partial_x\eta_0^{\varepsilon})^2  \\
    & +\eta_0^{\varepsilon}\partial_x^2 \eta^{\varepsilon}_0+\frac{2(\partial_x\omega_0^{\varepsilon})^2}{\eta_0^{\varepsilon}}+\frac{2\omega_0^{\varepsilon}}{\eta_0^{\varepsilon}}\partial_x^2 \omega_0^{\varepsilon}-\frac{2\omega_0^{\varepsilon}}{(\eta_0^{\varepsilon})^2 } \partial_x \eta^{\varepsilon}_0\partial_x \omega^{\varepsilon}_0-\varepsilon\partial_x^4\eta_0^{\varepsilon}, \\
\end{align*}
which, comparing to \eqref{initial_data_time}, implies that we need one more degree of space regularity for $\eta^{\varepsilon}_0-1$ than that of $\eta_0-1.$ 
Analogously to \eqref{est:energy}, one can show
\[
\mathfrak E^{\varepsilon}(t)\leq 2\mathfrak E^{\varepsilon}(0).
\]
This in turn implies
\[
\mathfrak E^{\varepsilon}(t)\lesssim \norm{\partial_t \eta_0^{\varepsilon}}{H^1}^2+\norm{\partial_t^2 \eta_0^{\varepsilon}}{L^2}^2+\norm{\omega_0^{\varepsilon}, \phi_0^{\varepsilon}}{H^4}^2\lesssim \varepsilon^2 \norm{\eta_0^{\varepsilon}-1}{H^4}^2+\norm{\eta_0^{\varepsilon}-1}{H^3}^2+\norm{\omega^{\varepsilon}_0,\phi_0^{\varepsilon}}{H^4}^2<\infty.
\]
Here $ \mathfrak E^\varepsilon $ is the energy for the approximating system \eqref{ShallowWaterMomentumgeneralized} in analogy to \eqref{intrmd-energy}. This finishes the proof. 
\end{proof}

\subsection{Existence of solutions}
\label{Ex-U}
With the uniform estimate \eqref{energy-viscous}, together with \eqref{Vsc_initial} and \eqref{Vsc_initial-2}, one can pass the limit $ \varepsilon \rightarrow 0 $ and, up to a subsequence, obtain
\begin{align}
    (\eta^\varepsilon - 1, ~ \omega^\varepsilon,~ \phi^\varepsilon) & \overset{*}{\rightharpoonup}  (\eta -1, ~ \omega, ~ \phi) & \text{weakly-$*$ in} ~ L^\infty([0,\mathfrak{T}];H^2(\mathbb R)),
\end{align}
for the same $\mathfrak{T}$ given in Proposition \ref{WellPosedViscousProblem}. By the Aubin-Lions compactness lemma, 
\begin{align}
\label{1-d-cnv:strong}
    (\eta^\varepsilon - 1, ~ \omega^\varepsilon, \phi^\varepsilon) & \rightarrow  (\eta -1, ~ \omega, ~ \phi) & \text{strongly in} ~ C([0,\mathfrak{T}];H^1_{loc}(\mathbb R)),
\end{align}
for some $ V = (\eta-1,\omega,\phi) \in L^\infty([0,\mathfrak{T}];H^2(\mathbb R)) $, satisfying 
\begin{equation}
        \label{energy-viscous-2}
        \sup_{0\leq t \leq T} \sum_{\substack{i + j \leq 2, \\
        i \geq 0, j \geq 0}} \norm{\partial_t^i\partial_x^j V(t)}{L^2} \leq C (\norm{\eta_0-1}{H^3} + \norm{\omega_0,\phi_0}{H^{4}} ),
    \end{equation}
thanks to \eqref{Vsc_initial} and \eqref{Vsc_initial-2}. Here $ H_{loc}^1(\mathbb R) $ representing the local $ H^1 $ space. That is, the strong convergence \eqref{1-d-cnv:strong} holds inside arbitrary compact subset of $ \mathbb R $. 

It is straightforward to verify that $ V $ solves \eqref{ShallowWaterV} with the above weak and strong compactness. 

\subsection{Uniqueness and continuous dependence on initial data}
\label{subsec:well-posedness-generalized}

Let $\mathcal V, \tilde{\mathcal{V}} \in L^{\infty}([0,\mathfrak{T}];H^2(\mathbb{R}))$ be two solutions to \eqref{ShallowWaterV} with initial data $\mathcal V(0,x) = \mathcal V_0:= (\eta_0-1,\omega_0,\phi_0), \tilde{\mathcal{V}}(0,x) = \tilde{\mathcal V}_0 :=(\tilde{\eta}_0-1,\tilde{\omega}_0,\tilde{\phi}_0)\in H^3(\mathbb{R}) \times H^4(\mathbb{R}) \times H^4(\mathbb{R})$, respectively. Then $\delta\mathcal{V}:=\mathcal{V}-\tilde{\mathcal{V}}=(\delta\eta,\delta\omega,\delta\phi)^\top$ solves
\[
\begin{cases}
A_0(\tilde{\mathcal{V}})\partial_t\delta\mathcal{V}+A_1(\tilde{\mathcal{V}}) \partial_x \delta\mathcal{V}+A_2 \partial_x^2 \delta \mathcal{V}=\begin{pmatrix}
    0 \\ 
    N(\mathcal{V},\tilde{\mathcal{V}}) \\
    0 \\
\end{pmatrix},\\
    \delta\mathcal{V}(0,x)=\delta\mathcal{V}_0=(\eta_0-\tilde{\eta}_0,\omega_0-\tilde{\omega}_0,\phi_0-\tilde{\phi}_0),
\end{cases}
\]
with
\[
N(\mathcal{V},\tilde{\mathcal{V}}):=\frac{\omega^2}{\eta^3} \phi-\frac{\tilde{\omega}^2}{\tilde{\eta}^3} \tilde{\phi}-\underbrace{[A_0(\mathcal{V})-A_0(\tilde{\mathcal V})]}_{\in L^{2}}\underbrace{\partial_t \mathcal V}_{\in L^\infty}-\underbrace{[A_1(\mathcal{V})-A_1(\tilde{\mathcal V})]}_{\in L^{2}}\underbrace{\partial_x \mathcal V}_{\in L^\infty}.
\]
Multiplying by $\delta\mathcal{V}^\top$ and integrating in space, the $L^2$ energy estimate for $\delta \mathcal V $ reads
\begin{equation*}
    \int \delta\mathcal{V}^\top A_0(\tilde{V})\partial_t \delta\mathcal{V}+\delta\mathcal{V}^\top A_1(\tilde{V}) \partial_x \delta\mathcal{V}+\delta\mathcal{V}^\top A_2 \partial_x^2 \delta\mathcal
    V dx=\int(\omega-\tilde{\omega}) N(\mathcal{V},\tilde{\mathcal{V}}) dx.
\end{equation*}
Denote $\tilde{\mathcal{E}}(t)$ the sum of the energy functionals for $\mathcal{V}$ and $\tilde{\mathcal{V}}$. A direct estimate yields
\begin{align*}
& \frac{1}{2}\frac{d}{dt}\norm{\delta\mathcal{V}}{L^2}^2\lesssim \tilde{\mathcal{P}}(\tilde{\mathcal{E}}(t)^{\frac{1}{2}}) \norm{\delta \mathcal{V}}{L^2}^2,
\end{align*}
for some positive polynomial $\tilde{\mathcal{P}}$. Now Grönwall's inequality implies 
\begin{equation}
\label{ContDepend}
\norm{\delta\mathcal{V}}{L^2}^2\lesssim \exp \left(\int_0^{\mathfrak{T}} \tilde{\mathcal{P}}(\tilde{\mathcal{E}}(t)^{\frac{1}{2}}) dt\right) \norm{\delta\mathcal{V}_0}{L^2}^2.
\end{equation}
In particular, if $\mathcal{V}_0=\tilde{\mathcal{V}}_0$, then $\mathcal{V}\equiv \tilde{\mathcal{V}}.$ This finishes the proof of well-posedness of the generalized system \eqref{ShallowWaterV}.

\section{Well-posedness of the 1d shallow water system \eqref{ShallowWaterMomentumConsice}}
\label{Sect:WP}

We now are ready to establish well-posedness system \eqref{ShallowWaterMomentumConsice} for initial data with regularity $(h_0-1,m_0)\in H^3(\mathbb R)\times H^2 (\mathbb R)$. To do so, we need elliptic estimates and energy estimates for the original shallow water system \eqref{ShallowWaterMomentumConsice}. We emphasize that the elliptic estimates are not available for the generalized system \eqref{ShallowWaterV}; see Remark \ref{rm:non-elliptic-est-general-sys}, below. 
\subsection{Elliptic estimates} We define the evolutionary energy 
\begin{equation}
    \label{EnergyEvolution}
    \mathcal F(t) := \sum_{0\leq i\leq 1} \norm{\partial_t^i (h-1),\frac{\partial_t^i m}{\sqrt h}, \partial_t^i \partial_x h}{L^2}^2,
\end{equation}
and the elliptic energy 
\begin{equation}
\label{EnergyElliptic}
    \mathfrak F(t):=\sum_{0\leq j\leq 2} \Vert \partial_x ^j (h-1),\frac{\partial_x^j m}{\sqrt{h}},\partial_x^{j}\partial_x h\Vert_{L^2}^2.
\end{equation}
Similar to \eqref{non-degenerate}, we make the {\it a priori} assumption for $ h $:
\begin{equation}
    \label{non-degenerate-h}
  0<\frac{1}{2}\underline h \leq \inf_{x,t} h(x,t) \leq \sup_{x,t} h(x,t) \leq 2 \overline h.
\end{equation}
We make the following observation:
\begin{lemma}
    \label{EllipticLemma} 
    Under the {\it a priori} assumption \ref{non-degenerate-h},
   there exists a positive polynomial $\mathcal{G}_1$ 
   such that
   \begin{equation}
   \label{EllipticEstimate1}
       \mathfrak F(t)\leq \mathcal{G}_1(\mathcal{F}(t)).
   \end{equation}

\begin{proof}
For the original system \eqref{ShallowWaterMomentumConsice}, we have
        \begin{gather}
        \label{EllipticHeight}
            \norm{\partial_x m}{L^2}= \norm{\partial_t h}{L^2}, \\
            \label{EllipticHeightD}
            \norm{\partial_x^2 m}{L^2}=\norm{\partial_t \partial_x h}{L^2}.
        \end{gather}
    At the same time, applying integration by parts, one has that
    \begin{equation}
            \label{EllipticMomentum}
    \begin{gathered}
            \norm{\partial_x^3 h}{L^2}^2+\norm{\partial_x h}{L^2}^2 + 2 \norm{\partial_x^2 h}{L^2}^2 = \norm{-\partial_{x}^3h + \partial_x h}{L^2}^2\\
            = \norm{\frac{1}{h}\partial_t m + \frac{2 m \partial_x m}{h^2} - \frac{m^2 \partial_x h}{h^3}}{L^2}^2 \lesssim \norm{\partial_t m}{L^2}^2+\norm{m}{H^1}^4+\norm{m}{H^1}^4 \norm{\partial_x h}{L^2}^2,
        \end{gathered}
    \end{equation}
    where we have applied \eqref{non-degenerate-h}.
This finishes the proof of \eqref{EllipticEstimate1}.

\end{proof}
\end{lemma}
\begin{remark}
\label{rm:non-elliptic-est-general-sys}
    Lemma \ref{EllipticLemma} does not hold for the generalized system \eqref{ShallowWaterV}. Indeed, in the $(\eta-1,\omega,\phi)$-variable, the corresponding estimate \eqref{EllipticMomentum} becomes 
    \[
    \norm{-\partial_x^2 \phi+\partial_x h}{L^2}^2 \lesssim \norm{\partial_t \omega}{L^2}^2+\norm{\omega}{H^1}^4+\norm{\omega}{H^1}^2\norm{\phi}{L^2}.
    \]
    In particular, since $ \phi \neq \partial_x h $ in general, one cannot obtain the higher order regularity of $ h $ from the left hand side of the above estimate.

    By making use of the elliptic estimate \eqref{EllipticEstimate1} for the shallow water system \eqref{ShallowWaterMomentumConsice}, we will be able to obtain sharper regularity for the solutions.
    
\end{remark}
\subsection{Energy estimates} For the shallow water system \eqref{ShallowWaterMomentumConsice}, we have the following analog to Lemma \ref{LemmaPhi}:
\begin{lemma}
\label{EnergyOriginal}
    Let $\Phi:\mathbb R \times\mathbb R\to\mathbb{R}^3$ be the solution to the linear equation
    \[
    A_0(U)\partial_t \Phi +A_1(U)\partial_x \Phi+A_2 \partial_x^2 \Phi=G.
    \]
    Then for a.e. $t$, we have
    \begin{equation}
        \frac{d}{dt} (A_0(U)\Phi,\Phi)\leq C_{\underline{h},\overline{h}} ( \mathcal{F}(t) + \mathfrak F(t) ) \norm{\Phi}{L^2}^2+2\norm{G}{L^2} \norm{\Phi}{L^2}.
    \end{equation}
\end{lemma}
The proof of Lemma \ref{EnergyOriginal} follows with the same arguments as in Lemma \ref{LemmaPhi}.

\smallskip 

One can similarly prove the analog of Proposition \ref{prop:a-priori-V} for $\mathcal{F}.$ That is:
\begin{proposition}
    \label{apriori-Original}
    The solution to \eqref{ShallowWaterMomentumConsice} satisfies the uniform bound
    \[
    \sup_{0\leq t\leq {T}} \mathcal{F}(t)\leq 2 \mathcal{F}(0),
    \]
    for some $T\in (0,\infty)$ depending only on the initial data and $\underline h$, $\overline{ h}.$
\end{proposition}

\begin{proof}
    Proposition \ref{apriori-Original} follows from an argument similar to that in Section \ref{1DEnergyEstimates}. Here we only sketch the main steps. We only consider, for the shallow water system \eqref{ShallowWaterMomentumConsice}, 
    \begin{equation}
    \label{sys:dd-Original}
    A_0(U)\partial_t U_t+A_1(U)\partial_x U_t+A_2\partial_{x}^2 U_t=\mathfrak G_t,
    \end{equation}
    where 
    \begin{align*}
        & \mathfrak G_t:=\partial_t S(U)-\partial_t A_0(U) \partial_t U-\partial_t A_1(U)\partial_x U.
    \end{align*}
    Applying Lemma \ref{EnergyOriginal} to systems \eqref{ShallowWaterMomentumConsice} and \eqref{sys:dd-Original}, we obtain
\begin{equation}
    \label{h2:001-original}
    \frac{d}{dt} \mathcal F(t) \leq C_{\overline {h},\underline{h}} (\mathcal{F}(t) + \mathfrak F(t)) \norm{U,U_{t}}{L^2}^2 +2\norm{{S(U)},\mathfrak{G}_{t}}{L^2}\norm{U,U_{t}}{L^2} \leq \mathcal G_2(\mathcal F(t)),
\end{equation}
for some positive polynomial $ \mathcal G_2 $, thanks to Lemma \ref{EllipticLemma} and embedding inequalities. Integrating \eqref{h2:001-original} finishes the energy estimate. Moreover, following the argument as in Section \ref{Justification}, one can establish the justification of \eqref{non-degenerate-h}. This finishes the proof. 
\end{proof}
\subsection{Proof of Theorem \ref{thm:1-d}}
By taking $\eta=h$, $\omega=m$, and $\phi=\partial_x h$ in system \eqref{ShallowWaterV}, for $(h_0-1,m_0)\in H^5(\mathbb{R})\times H^4(\mathbb{R})$, Proposition \ref{EnergyProposition} implies that there exists a unique solution $ (h-1, m)\in H^3(\mathbb{R})\times H^2(\mathbb{R}) $ to the shallow water system \eqref{ShallowWaterMomentumConsice}.

In general, for $ (h_0-1, m_0) \in H^3(\mathbb R) \times H^2(\mathbb R) $, one chooses a sequence of initial data $ (h_{k,0} - 1, m_{k,0}) \in H^5(\mathbb R) \times H^4(\mathbb R)$, $ k \in \mathbb N $, such that
\begin{equation}
\label{approximation_initial_k}
\begin{aligned}
    h_{k,0} - 1 & \rightarrow h_0 - 1  & \text{in} \ H^3(\mathbb R),\\
    m_{k,0} & \rightarrow m_0 & \text{in} \ H^2(\mathbb R), 
\end{aligned}
\end{equation}
and the initial data
\begin{equation}
    \label{V_k}
    V_{k,0} := \begin{pmatrix}
        h_{k,0} - 1 \\ m_{k,0} \\ \partial_x h_{k,0}
    \end{pmatrix} \in H^3(\mathbb{R}) \times H^4(\mathbb{R}) \times H^4(\mathbb{R}). 
\end{equation}
The well-posedness theory of the generalized system \eqref{ShallowWaterV} implies that for each fixed $k=1,2,\ldots$, there exists $ T_k \in (0,\infty) $ such that we have a unique solution $ V_k :=( h_k -1, m_k, \phi_k )^\top $. 
In particular, $ (h_k-1, m_k) $ is a solution to the shallow water system \eqref{ShallowWaterMomentumConsice}. 

Applying Proposition \ref{apriori-Original} and Lemma \ref{EllipticLemma}, one obtains the following estimate:
\begin{equation}
\label{EllipticUniform}
\begin{gathered}
\sup_{0\leq t \leq \mathfrak{T}'} \bigl\lbrace 
\norm{h_k(t)-1}{H^3}^2+\norm{m_k(t)}{H^2}^2 + \norm{h_k(t)-1, \partial_t h_k(t)}{H^1}^2 + \norm{m_k(t), \partial_t m(t)}{L^2}^2 \bigr\rbrace \\
\lesssim \norm{\partial_t m_{k,0},m_{k,0}}{L^2}^2+\norm{\partial_t h_{k,0}, h_{k,0}- 1}{H^1}^2 \lesssim  \norm{h_{k,0}-1}{H^3}^2 + \norm{m_{k,0}}{H^2}^2,
\end{gathered}
\end{equation}
for some sufficiently short time $\mathfrak{T}'>0$ independent of $k$. The last inequality follows from directly using the equations \eqref{ShallowWaterMomentumConsice} to obtain the initial data for the time derivatives, similar to \eqref{initial_data_time}.
Therefore, by the convergence \eqref{approximation_initial_k} and the bound \eqref{EllipticUniform}, passing the limit $ k \rightarrow \infty $, 
one obtains 
the unique solution $(h-1,m)\in L^{\infty}([0,\mathfrak{T}'];H^3(\mathbb R))\times L^{\infty}([0,\mathfrak{T}'];H^2(\mathbb R))$ to the shallow water system \eqref{ShallowWaterMomentumConsice}. The uniqueness and continuous dependence on the initial data follow with arguments similar to those in Section \ref{subsec:well-posedness-generalized}. This concludes the proof of Theorem \ref{thm:1-d} after recalling $u=\frac{m}{h}$.

\section{{\it A priori} estimate for two dimensional flow}
\label{sec:2d} 

\subsection{Symmetric reformulation}
\label{subsec:2d-reformulation}

We will turn our focus to the case of multi-dimensions in system \eqref{ShallowWaterEquations}. In this section, we consider the case when $ n = 2 $, i.e., the two dimensional flow. To be more precise, we will only establish the {\it a priori} estimate using the symmetric structure. The local well-posedness follows via similar arguments as in Sections \ref{vv}--\ref{Sect:WP}.

For $ \mathbb R^2 = \lbrace (x_1, x_2) \rbrace $, we write, for $j =1,2 $,
\begin{equation}
    \label{def:2d-drvt}
    \partial_j := \partial_{x_j}.
\end{equation}
We write, for any vector field $ v = (v_1, v_2)^\top $, 
\[
v^{\perp}:=\begin{pmatrix}
    -v_2 \\
    v_1
\end{pmatrix} = J v, \quad\nabla^{\perp}:= \begin{pmatrix}
    -\partial_2 \\
    \partial_1 \\
\end{pmatrix} = J \nabla ,
\]
where
\begin{equation}
\label{def:anti-matrix}
    J:= \begin{pmatrix}
        0 & -1 \\ 1 & 0
    \end{pmatrix},
\end{equation}
and we consider the 2d flow with the non-degenerate condition
\begin{equation}
    \label{non-degenerate-nD}
    0<\frac{1}{2}\underline h \leq \inf_{x,t} h(x,t) \leq \sup_{x,t} h(x,t) \leq 2 \overline h < \infty. 
\end{equation}

Similarly as before, consider the momentum variable $ m = (m_1,m_2)^\top:= hu = h(u_1,u_2)^\top $. Recall that the shallow water system \eqref{ShallowWaterEquations} in 2d can be written as in \eqref{ShallowWaterMomentum}. We point out, that unlike the 1d case as in \eqref{ShallowWaterMomentumConsice}, the 2d flow \eqref{ShallowWaterMomentum} is not a symmetric system. Indeed, by writing $ U:= (h -1, m, \nabla h)^\top = (h - 1, m_1, m_2, \partial_1 h, \partial_2 h) \in \mathbb R^5 $, one has
\begin{equation}
\label{ShallowWaterMomentum2D}
\diag\left(1,\frac{1}{h},\frac{1}{h},1,1\right) \partial_t U+M_1 \partial_1 U +M_2\partial_2 U +M_3\partial_1^2 U +M_4\partial_2^2 U +M_5\partial_{1}\partial_2 U=S,
\end{equation}
where
\begin{equation}
    \label{def:M-2d}
\begin{aligned}
M_1 & :=\begin{pmatrix}
    0 & 1 & 0 & 0 & 0 \\
    1 & \frac{2m_1}{h^2} & 0 & 0 & 0 \\
    0 & {\frac{m_2}{h^2}} & \frac{m_1}{h^2} & 0 & 0 \\
    0 & 0 & 0 & 0 & 0 \\
    0 & 0 & 0 & 0 & 0 \\
\end{pmatrix}, & M_2 &:=
\begin{pmatrix}
    0 & 0 & 1 & 0 & 0 \\
    0 & \frac{m_2}{h^2} & {\frac{m_1}{h^2}} & 0 & 0 \\
    1 & 0 & \frac{2m_2}{h^2} & 0 & 0 \\
    0 & 0 & 0 & 0 & 0 \\
    0 & 0 & 0 & 0 & 0 \\
\end{pmatrix}, \\
M_3& :=\begin{pmatrix}
    0 & 0 & 0 & 0 & 0 \\
    0 & 0 & 0 & -1 & 0 \\
    0 & 0 & 0 & 0 & 0 \\
    0 & 1 & 0 & 0 & 0 \\
    0 & 0 & 0 & 0 & 0 \\
\end{pmatrix},
&
M_4 &:=\begin{pmatrix}
    0 & 0 & 0 & 0 & 0 \\
    0 & 0 & 0 & 0 & 0 \\
    0 & 0 & 0 & 0 & -1 \\
    0 & 0 & 0 & 0 & 0 \\
    0 & 0 & 1 & 0 & 0 \\
\end{pmatrix}, \\
M_5 &:=\begin{pmatrix}
    0 & 0 & 0 & 0 & 0 \\
    0 & 0 & 0 & 0 & -1 \\
    0 & 0 & 0 & -1 & 0 \\
    0 & 0 & 1 & 0 & 0 \\
    0 & 1 & 0 & 0 & 0 \\
\end{pmatrix}, & S&:=\begin{pmatrix}
    0 \\
    \frac{1}{h^3} m_1 (m_1 \partial_1 h + m_2 \partial_2 h) \\
    \frac{1}{h^3} m_2 (m_1 \partial_1 h + m_2 \partial_2 h)  \\
    0 \\
    0 \\
\end{pmatrix}.
\end{aligned}
\end{equation}
In particular, $ M_1 $ and $ M_2 $ are not symmetric matrices. Therefore, directly applying energy estimate as in Section \ref{1DEnergyEstimates} will yield a loss of regularity. To remedy this, we isolate the non-symmetric part of system \eqref{ShallowWaterMomentum2D} by writing
\begin{equation}
    \label{def:2d-non-sym}
    K:= B_1 \partial_1 m + B_2 \partial_2 m
\end{equation}
with
\begin{equation}
    \label{def:non-sym-matrix-2d}
    B_1 := \begin{pmatrix}
        \frac{2m_1}{h^2} & 0 \\
        \frac{m_2}{h^2} & \frac{m_1}{h^2}
    \end{pmatrix}, \quad B_2 := \begin{pmatrix}
        \frac{m_2}{h^2} & \frac{m_1}{h^2} \\
        0 & \frac{2m_2}{h^2}
    \end{pmatrix}.
\end{equation}
Now, write 
\begin{equation}
    \label{2d-001}
    B_1 =  \underbrace{\begin{pmatrix}
        \frac{2m_1}{h^2} & \frac{m_2}{2h^2} \\
        \frac{m_2}{2h^2} & \frac{m_1}{h^2}
    \end{pmatrix}}_{:= B_\mathrm{1,sym}} + \frac{m_2}{2h^2} J, \qquad B_2 =  \underbrace{\begin{pmatrix}
        \frac{m_2}{h^2} & \frac{m_1}{2h^2} \\
        \frac{m_1}{2h^2} & \frac{2m_2}{h^2}
    \end{pmatrix}}_{:= B_\mathrm{2,sym}} - \frac{m_1}{2h^2}J.
\end{equation}
Here $ J $ is the anti-symmetric matrix given in \eqref{def:anti-matrix}.
Then one can write $ K $ from \eqref{def:2d-non-sym} as
\begin{equation}
    \label{2d-K-good-form}
    K = B_\mathrm{1,sym} \partial_1 m + B_\mathrm{2,sym} \partial_2 m + \underbrace{\frac{m_2}{2h^2} J \partial_1 m - \frac{m_1}{2h^2} J \partial_2 m}_{=:K'}.
\end{equation}
We can rewrite
\begin{equation}
    \label{2d-002}
    \begin{gathered}
    K' = \frac{m_2}{2h^2} J \nabla m_1 - \frac{m_1}{2h^2} J \nabla m_2 + \frac{m_2}{2h^2}(J\partial_1 m - J\nabla m_1) + \frac{m_1}{2h^2} (J\nabla m_2 - J\partial_2 m) \\
    = \frac{m_2}{2h^2} \nabla^\perp m_1 - \frac{m_1}{2h^2} \nabla^\perp m_2 + \frac{m_2}{2h^2} J (\nabla_\mathrm{skw}m)_1 - \frac{m_1}{2h^2} J (\nabla_\mathrm{skw}m)_2,
    \end{gathered}
\end{equation}
where 
\begin{equation}
    \label{2d-003}
    \nabla_\mathrm{skw} m := \nabla m - (\nabla m)^\top = (\partial_1 m - \nabla m_1, \partial_2 m - \nabla m_2) = \begin{pmatrix}
        0 & - \curl m \\
        \curl m & 0
    \end{pmatrix}
\end{equation}
is the anti-symmetric gradient of $ m $, and $(\nabla_\mathrm{skw} m)_i $
is the $ i $-th column of $\nabla_{\mathrm{skw}}m$. Therefore, $ K' $, written in the form of \eqref{2d-002}, is the sum of the vorticity (the $\nabla_\mathrm{skw} m$ terms) and the dual vorticity (the $J\nabla m_j = \nabla^\perp m_j$, $ j = 1,2$ terms). Thus, the shallow water equations in 2D \eqref{ShallowWaterMomentum2D} can be written as
\begin{equation}
    \label{ShallowWaterMomentum2D-1}
    M_0\partial_t U+M_1' \partial_1 U +M_2' \partial_2 U +M_3\partial_1^2 U +M_4\partial_2^2 U +M_5\partial_{1}\partial_2 U=S + S',
\end{equation}
where $ M_3, M_4, M_5, S $ are given as in \eqref{def:M-2d} and
\begin{equation}
    \label{def:M-2d-1}
    \begin{aligned}
    M_0& :=\diag\left(1,\frac{1}{h},\frac{1}{h},1,1\right), \quad    M_1' :=\begin{pmatrix}
        0 & 1 & 0 & 0 & 0 \\
        1 & \frac{2m_1}{h^2} & {\frac{m_2}{2h^2}} & 0 & 0 \\
        0 & {\frac{m_2}{2h^2}} & \frac{m_1}{h^2} & 0 & 0 \\
        0 & 0 & 0 & 0 & 0 \\
        0 & 0 & 0 & 0 & 0 \\
    \end{pmatrix}, \qquad M_2' :=
    \begin{pmatrix}
        0 & 0 & 1 & 0 & 0 \\
        0 & \frac{m_2}{h^2} & {\frac{m_1}{2 h^2}} & 0 & 0 \\
        1 & {\frac{m_1}{2 h^2}} & \frac{2m_2}{h^2} & 0 & 0 \\
        0 & 0 & 0 & 0 & 0 \\
        0 & 0 & 0 & 0 & 0 \\
    \end{pmatrix}, \\
    S' & := \begin{pmatrix}
        0 \\
        - K' \\
        0 \\
        0 \\
    \end{pmatrix} = \begin{pmatrix}
        0 \\
        - \frac{m_2}{2h^2} \nabla^\perp m_1 + \frac{m_1}{2h^2} \nabla^\perp m_2 - \frac{m_2}{2h^2} J (\nabla_\mathrm{skw}m)_1 + \frac{m_1}{2h^2} J (\nabla_\mathrm{skw}m)_2 \\
        0 \\
        0 \\
    \end{pmatrix}.
    \end{aligned}
\end{equation}

\smallskip 

Meanwhile, unlike the one dimensional flow, we will need to track the evolution of the vorticity to handle $ S' $ in the following. Applying $\curl$ to \eqref{MoMeNtUm} yields
\begin{equation}
    \label{Curled}
    \partial_t \curl u+u\cdot \nabla \curl u = - \Div u \curl u \overset{\eqref{MaSs}}{=} \frac{1}{h}(\partial_t h + u \cdot \nabla h) \curl u.
\end{equation}
Therefore, dividing \eqref{Curled} with $ h $ and re-arranging the resultant equation lead to the specific vorticity formula \cite{BreschDesjardinsMetivier,ShkollerVicol,ChenCialdeaShkollerVicol},
\begin{equation}
    \label{CurledTransport}
    \partial_t \theta +u\cdot \nabla \theta =0,
\end{equation}
where
\begin{equation}
    \label{ThetaVariable}
    \theta:=\frac{\curl u}{h}.
\end{equation}

\smallskip 
In addition, 
consider the energy functional
\begin{equation}
    \label{Energy-2d}
    \mathcal{E}_2(t):=\sum_{\substack{i+ j\leq 3, \\ i,j \geq 0}} \norm{\partial_t^i \nabla^{j}(h-1),\partial_t^i\nabla^j m,\partial_t^i\nabla^{j+1} h}{L^2}^2,
\end{equation}
and the intermediate energy 
\begin{equation}
    \label{intrmd-energy-2d}
    \mathfrak{E}_2(t):=\sum_{\substack{i+ j \leq 3, \\ i, j \geq 0}} (M_0\partial_t^i\nabla^j  U(t),\partial_t^i\nabla^j U(t)).
\end{equation}
By the non-degenerate condition \eqref{non-degenerate-nD}, there exists a constant $C_{\underline{h},\overline{h}}\in (0,\infty)$ such that
       \begin{equation}
       \label{Energy-equiv-2d-cauchy}
            \frac{1}{C_{\underline{h},\overline{h}}}\mathcal{E}_2(t)\leq \mathfrak{E}_2(t)\leq C_{\underline{h},\overline{h}}\mathcal{E}_2(t).
        \end{equation}

\subsection{Vorticity estimate}
\label{subsec:2d-vorticity}

Let $ \partial \in \lbrace \partial_1, \partial_2 \rbrace $ be the spatial derivative.  
For $ i + j \leq 3 $, and $ i \geq 0, \ j \geq 0 $, consider the evolution of $ \theta_{ij}:=\partial_t^i \partial^j \theta $ from \eqref{CurledTransport}, 
\begin{equation}
    \label{eq:dd-curl-2d}
    \partial_t \theta_{ij} + u \cdot \nabla \theta_{ij} = - \sum_{\substack{1\leq a \leq i, \\ 1 \leq b \leq j}} \partial_t^{a} \partial^{b}u \cdot \nabla \partial_t^{i-a}\partial^{j-b} \theta,
\end{equation}
where we have omitted the combinatorial constant on the right hand side. 
Then the standard $ L^2 $ estimate of $ \theta_{ij} $ yields
\begin{equation}
    \label{est:theta-ij}
    \begin{gathered}
    \dfrac{d}{dt}\norm{\theta_{ij}}{L^2}^2 \lesssim \norm{ \Div u}{L^\infty} \norm{\theta_{ij}}{L^2}^2 + \sum_{\substack{1\leq a \leq i, \\ 1 \leq b \leq j}} \norm{\partial_t^{a} \partial^{b}u \cdot \nabla \partial_t^{i-a}\partial^{j-b} \theta}{L^2} \norm{\theta_{ij}}{L^2}\\
    \lesssim \norm{u}{H^3} \norm{\theta_{ij}}{L^2}^2 + \sum_{\substack{i+j\leq 3, \\ i,j \geq 0}}\norm{\partial_t^i\partial^j u}{L^2} \sum_{\substack{i+j\leq 3, \\ i,j \geq 0}} \norm{\partial_t^i \partial^j \theta}{L^2}^2.
    \end{gathered}
\end{equation}
Applying Gr\"onwall's inequality the yields
\begin{equation}
    \label{est:theta-2}
\sum_{\substack{i+j\leq 3, \\ i,j \geq 0}} \norm{\partial_t^i \nabla^j \theta (t)}{L^2}  \lesssim \sum_{\substack{i+j\leq 3, \\ i,j \geq 0}} \norm{\partial_t^i \nabla^j \frac{\curl u (0)}{h(0)}}{L^2} \exp (t \sup_{0\leq s \leq t}  \sum_{\substack{i+j\leq 3, \\ i,j \geq 0}}\norm{\partial_t^i\nabla^j u(s)}{L^2} ).
\end{equation}

Meanwhile, since, $ u = \frac{m}{h} $ and, with direct calculation, 
\begin{equation}
\label{eq:theta-2}
\theta=\frac{\curl u}{h}= \frac{\curl m}{h^2}-\frac{m\cdot \nabla^{\perp}h}{h},
\end{equation}
the following lemma follows with a straightforward calculation:
\begin{lemma}[Vorticity estimate]
\label{lm:vorticity}
For a solution of \eqref{ShallowWaterMomentum2D} (or equivalently, of \eqref{ShallowWaterMomentum2D-1}) that satisfies \eqref{non-degenerate-nD}, one has that, $ \forall \ t\in (0,\infty) $,  
\begin{equation}
    \label{est:vorticity}
    \sum_{\substack{i+j\leq 3, \\ i,j \geq 0}} 
    \norm{\partial_t^i \nabla^j \curl m (t)}{L^2} 
    \lesssim \sup_{0\leq s \leq t} \mathcal P(\mathcal E_2(s)) + C_\mathrm{skw} 
    \exp(t \sup_{0\leq s \leq t} \mathcal P(\mathcal E_2(s))),
\end{equation}
where $ C_\mathrm{skw}\in(0,\infty) $ is a constant depending on 
\begin{equation}
    \label{2d-initial-bound}
    \sum_{\substack{i+j\leq 3,\\i,j \geq 0}} \norm{\partial_t^i \nabla^j \curl u(t=0),\partial_t^i \nabla^j u(t=0) , \partial_t^i \nabla^j h(t=0),\partial_t^i \nabla^{j+1} h(t=0)}{L^2} < \infty. 
\end{equation}
\end{lemma}


\subsection{Energy estimate}
\label{subsec:2d-energy-est}

Now we outline the \emph{a priori} energy estimates for \eqref{ShallowWaterMomentum2D-1}. For simplicity we stay in the original variables $U = (h-1,m,\nabla h)^\top.$ 
The left-hand-side of \eqref{ShallowWaterMomentum2D-1} has the correct symmetric structure, and the first source term $S$ is of lower order. Hence, it remains to check the energy estimates of the source term $S'$. Heuristically, this is done by making the following crucial observation: Thanks to \eqref{2d-003}, $ \nabla_\mathrm{skw} m $ can be controlled by the vorticity estimate \eqref{est:vorticity}, and, with respect to the $ L^2 $ inner product, $ \nabla^\perp $ is the duality operator of $ \curl $, i.e., 
\begin{equation}
\label{duality-formula}
    \int \nabla^\perp f \cdot g \,d\vec x = - \int f \curl g \,d\vec x
\end{equation}
for $ f $ and $ g $ being scaler and vector valued functions, respectively. \eqref{duality-formula} is going to help us to close the energy estimates.
\smallskip

We will only present the high order estimate. From \eqref{ShallowWaterMomentum2D-1}, one can write down
the systems for $U_{ij}:=\partial_t^i\partial^j U$, $ i + j = 3 $, $ i,j \geq 0 $,
\begin{equation}
\label{sys:dd-2d}
  \underbrace{M_0\partial_t U_{ij} +M_1'\partial_1 U_{ij}+M_2'\partial_2 U_{ij}+M_3\partial_1^2 U_{ij} +M_4 \partial_2^2 U_{ij}+M_5\partial_1\partial_2 U_{ij} =\mathcal{G}_{ij}}_{\text{symmetric part}}+\underbrace{\mathcal{G}'_{ij}}_{\mathclap{\text{non-symmetric part}}}, \\
\end{equation}
where 
\begin{equation}
    \begin{aligned}
      &  \mathcal{G}_{30}:= \partial_t^3 S-3\partial_t M_0\partial_t^3U-3\partial_t^2 M_0\partial_t^2 U-\partial_t^3 M_0\partial_tU-3\partial_tM_1\partial_t^2\partial_1 U-3\partial_t^2 M_1\partial_t \partial_1 U\\
      & -\partial_t^3 M_1 \partial_1 U-3\partial_t M_2\partial_t^2 \partial_2 U-3\partial_t^2 M_2\partial_t \partial_2U-\partial_t^3 M_2\partial_2U,\\
      & \mathcal{G}_{21}:= \partial_{t}^2 \partial  S-\partial M_0\partial_t^3 U-2\partial \partial_t M_0\partial_t^2 U-\partial \partial_t^2 M_0\partial_t U-\partial_t^2 M_0\partial_t\partial  U-2\partial_tM_0\partial_t^2 \partial  U \\
      &-\partial M_1\partial_t^2\partial_1 U-2\partial \partial_t M_1\partial_t\partial_1 U -\partial \partial_t^2M_1\partial_1 U-\partial_t^2 M_1\partial_1\partial  U-2\partial_tM_1\partial_t\partial_1\partial  U \\
      &-\partial  M_2\partial_t^2\partial_2 U-2\partial \partial_t M_2 \partial_t\partial_2 U- \partial  \partial_t^2 M_2\partial_2 U-\partial_t^2 M_2\partial_2\partial U-2\partial_tM_2\partial_t\partial_2\partial  U,\\
      & \mathcal{G}_{12}:= \partial_t \partial^2 S-\partial M_0\partial_t^2 \partial  U-2\partial \partial_t M_0\partial_t\partial  U-\partial^2 \partial_t M_0 \partial  U-\partial_t M_0\partial^2\partial_t U \\
      & -2\partial M_0\partial \partial_t^2 U-\partial M_1\partial_t\partial \partial_1 U-2\partial_t\partial  M_1\partial_1\partial  U-\partial^2 \partial_t M_1\partial_1 U-\partial_tM_1\partial^2\partial_1 U\\
      &-2\partial  M_1\partial \partial_t\partial_1 U-\partial M_2\partial_t\partial \partial_2 U-2\partial_t\partial  M_2\partial_2\partial  U-\partial^2 \partial_t M_2\partial_2 U\\
      &-\partial_tM_2\partial^2\partial_2 U-2\partial  M_2\partial \partial_t\partial_2 U, \\
      & \mathcal{G}_{03}:= \partial^3 S-3\partial  M_0\partial^2 \partial_t U-3\partial^2 M_0\partial \partial_t U -\partial^3 M_0\partial_t U-3\partial  M_1\partial^2\partial_1 U \\
      &-3\partial^2 M_1 \partial  \partial_1 U -\partial^3 M_1 \partial_1 U-3\partial  M_2\partial^2 \partial_2 U-3\partial^2 M_2\partial  \partial_2 U-\partial^3 M_2\partial_2 U,\\
    \end{aligned}
\end{equation}
and 
\begin{equation}
    \mathcal{G}'_{30}:=\partial_t^3 S', \quad \mathcal{G}_{21}':=\partial_t^2 \partial S', \quad \mathcal{G}'_{12}:=\partial_t \partial^2 S', \quad \mathcal{G}'_{03}:=\partial^3 S'.
\end{equation}
Here we have written \eqref{sys:dd-2d} into the symmetric part and the non-symmetric part, where the symmetric part can be handled similar to \eqref{sys:dd-V}--\eqref{est:nonlinear-006}. To be more precise, taking the $ L^2 $-inner product of \eqref{sys:dd-2d} with $ U_{ij} $, and applying integration by parts, H\"older's inequality, and Sobolev embedding inequality, lead to 
\begin{equation}
\label{est:2d-U-ij}
    \frac{1}{2}\dfrac{d}{dt}(M_0 U_{ij}, U_{ij}) \lesssim \mathcal P(\mathcal E_2(t)^{\frac{1}{2}}) \norm{U_{ij}}{L^2}^2 + (\mathcal G_{ij}', U_{ij}). 
\end{equation}

It remains to calculate $ (\mathcal G_{ij}', U_{ij}) $. We will again only focus on the high order terms. Let  $\partial_H\in \{\partial_t^3,\partial_t^2 \partial ,\partial_t\partial^2, \partial^3\}$. Then from \eqref{def:M-2d-1}, one can check that the high order integral in $ (\mathcal G_{ij}', U_{ij}) $ reads
\begin{equation}
\begin{aligned}
\label{2d-high-energy}
   & \int \frac{\partial_H m}{2h^2}\cdot \bigg[-m_2\nabla^{\perp}\partial_H m_1+m_1\nabla^{\perp} \partial_H m_2-m_2 \partial_HJ(\nabla_{\mathrm{skw}}m)_1+m_1\partial_H J(\nabla_{\mathrm{skw}}m)_2  \\
   & -\partial_H m_2 \nabla^{\perp}m_1+\partial_H m_1\nabla^{\perp}m_2-\partial_H m_2 J(\nabla_{\mathrm{skw}}m)_1+\partial_H m_1J(\nabla_{\mathrm{skw}}m)_2 \bigg]d\vec{x} \\
   & \overset{\eqref{duality-formula}}{=}\int \frac{\partial_Hm}{2h^2} \cdot \bigg[-m_2 \partial_HJ(\nabla_{\mathrm{skw}}m)_1+m_1\partial_H J(\nabla_{\mathrm{skw}}m)_2-\partial_H m_2 J(\nabla_{\mathrm{skw}}m)_1+\partial_H m_1J(\nabla_{\mathrm{skw}}m)_2 \bigg] d\vec{x} \\
   &+\int \frac{1}{2h^2} \left[ m_2\partial_H m_1\curl\partial_H m-m_1 \partial_H m_2 \curl\partial_H m \right] d\vec{x} + \mathrm{l.o.t} \\
   & \lesssim \mathcal E_2^2 \times (\norm{\curl \partial_t^3 m}{L^2}+\norm{\curl \partial_t^2 m}{H^1}+\norm{\curl \partial_t m}{H^2}+\norm{\curl m}{H^3}) + \mathrm{l.o.t},
\end{aligned}
\end{equation}
thanks to the duality \eqref{duality-formula}. 
Here the lower order integral $ \mathrm{l.o.t} $ can be treated similarly as before. 

\smallskip 

Collecting all estimates above as well as the lower order estimates, using \eqref{est:vorticity} from Lemma \ref{lm:vorticity}, we conclude that
\begin{equation}
    \label{energy-ineq-2d-3}
    \frac{d}{dt}\mathfrak{E}_2(t)\leq C_{\underline h,\overline{h}} \mathcal P(\mathfrak{E}_2(t)^{\frac{1}{2}}) \bigl\lbrack 1 + C_\mathrm{skw}
    \exp(t \sup_{0\leq s \leq t} \mathcal P(\mathfrak E_2(s)^{\frac{1}{2}})) \bigr\rbrack,
    \end{equation}
where we have used \eqref{Energy-equiv-2d-cauchy}.

Therefore, similarly to \eqref{initial_data_time} and \eqref{2d-initial-bound}, one can check that given sufficient regularity $(h_0-1,m_0)\in H^7(\mathbb{R})\times H^6(\mathbb{R}^2)$, we can close both the temporal and spatial energy estimates in dimension two. Hence, we have shown the following: 
\begin{proposition}
    Under the assumption \eqref{non-degenerate-nD}, there exists the solution to \eqref{ShallowWaterMomentum2D} satisfies
    \begin{equation}
        \sup_{0\leq t\leq T} \mathfrak{E}_2(t)\leq 2\mathfrak{E}_2(0),
        \end{equation}
        and, thanks to \eqref{Energy-equiv-2d-cauchy}, also satisfies        \begin{equation}
          \sup_{0\leq t\leq T}  \mathcal{E}_2(t)\leq C_{\underline{h},\overline{h}} \mathcal{E}_2(0),
        \end{equation}
        for some $T\in (0,\infty)$ and some $C_{\underline{h},\overline{h}}\in (0,\infty) $ depending only on the initial data and $\underline{h}$, $\overline{h}$.
\end{proposition}
Meanwhile, \eqref{non-degenerate-nD} can be closed following the arguments as in Proposition \ref{prop:a-priori-assumption}. Thus repeating the vanishing viscosity arguments as in Section \ref{vv}, we have shown the following: 
\begin{proposition}
\label{prop:2d-local-wellposedness}
    Consider initial data $(h_0-1,m_0)\in H^7(\mathbb{R}) \times H^6(\mathbb{R}^2)$. There exists some time $T\in (0,\infty)$ such that there exists a unique local-in-time solution $(h-1,m)\in L^{\infty}([0,T];H^4(\mathbb{R}))\times L^{\infty}([0,T];H^3(\mathbb{R}^2))$ to \eqref{ShallowWaterMomentum2D}, which satisfies
    \begin{equation}
        \sup_{0\leq t\leq T}\sum_{\substack{i+\vert\alpha\vert\leq 3, \\ i,\vert\alpha\vert\geq 0}} \norm{\partial_t^i\partial^{\alpha}(h-1),\partial_t^i\partial^{\alpha} m, \partial_t^i\partial^{\alpha}\nabla h}{L^2}^2 \leq C_{\underline{h},\overline{h}}\mathcal{E}_2(0),
    \end{equation}
    for some constant $C_{\underline{h},\overline{h}}\in (0,\infty)$ depending only on $\underline{h},$ $\overline{h}$. Moreover, following a calculation similar to \eqref{initial_data_time}, we have
    \begin{equation}
        \mathcal{E}_2(0)\lesssim \mathcal{P} (\norm{h_0-1}{H^7},\norm{m_0}{H^6}).
    \end{equation}
\end{proposition}

\section{Remark on the regularity \eqref{n-d-initial-regularity} for 2-dimensional flows}
\label{Regularity-remark}
The purpose of Section \ref{Regularity-remark} is to demonstrate the optimal regularity argument in Section \ref{Sect:WP} fails in dimension two. Namely, we cannot relax the regularity of the initial data from $(h_0 -1, u_0) \in H^7(\mathbb{R})\times H^6(\mathbb{R}^2)$ to $(h_0 -1, u_0) \in H^4(\mathbb{R})\times H^3(\mathbb{R}^2)$. The integer Sobolev index corresponding to $n=2$ is $s=3$: The corresponding evolution energy and elliptic energy are
\begin{equation}
    \label{EnergyEvolution-Odd}
    \mathcal{F}_{\text{bad}} (t):=\sum_{\substack{0\leq i\leq 1, \\ 0\leq i+\vert\alpha\vert\leq 2}}\norm{\partial_t^i\partial^{\alpha}(h-1),\frac{\partial_t^i \partial^{\alpha} m}{\sqrt{h}},\partial_t^i \partial^{\alpha}  \nabla h}{L^2}^2,
\end{equation}
and 
\begin{equation}
    \label{EnergyElliptic-Odd}
    \mathfrak{F}_{\text{bad}}(t):=\sum_{0\leq \vert \alpha \vert \leq 3} \norm{\partial^{\alpha}(h-1),\frac{\partial^{\alpha}m}{\sqrt{h}}, \partial^{\alpha}\nabla h}{L^2}^2,
\end{equation}
respectively. To integrate by parts in the energy estimates, one needs to compute, for $\vert \alpha \vert=1$,
\begin{equation}
    \partial_t\partial^{\alpha} \left(\frac{1}{h}\partial_t m\right) \partial_t\partial^{\alpha} m=[\underbrace{\partial^{\alpha} \left(\frac{1}{h}\right)}_{\in L^{\infty}} \boxed{\partial_t^2 m} ]\underbrace{\partial_t\partial^{\alpha} m}_{\in L^2} +\cdots .
\end{equation}
We require $\partial_t^2m\in L_x^2$: A direct calculation using \eqref{ShallowWaterMomentumMomentum} gives
\begin{equation}
\label{2-time-diff-momentum}
\begin{aligned}
   & \frac{1}{h}{\partial_t^2 m}+\cdots 
   ={\partial_t \nabla \Delta h}+\cdots . 
\end{aligned}
\end{equation}
However, the term $\partial_t \nabla\Delta h$ is not in the evolution energy \eqref{EnergyEvolution-Odd}. Consequently, $\mathcal{F}_{\text{bad}}$ and $\mathfrak{F}_{\text{bad}}$ cannot satisfy an inequality of the form \eqref{EllipticEstimate1} due to a loss of regularity. Hence, we conclude that it is not possible to loosen the regularity \eqref{n-d-initial-regularity} to the optimal regularity $(h_0-1,m_0)\in H^4(\mathbb{R})\times H^3(\mathbb{R}^2)$ without imposing additional regularity assumptions on $m_0$.
\begin{remark}
    We do not try to optimize the regularity as in Section \ref{Sect:WP} for dimensional higher than $ 1 $ hereafter.  
\end{remark}
\section{$ n $-dimensional irrotational flows, $ n \geq 3 $}
\label{sec:irrotational-flow}

For flows in dimensions $ 3 $ or larger, one can still write down the system \eqref{ShallowWaterEquations} as a sum of the symmetric and non-symmetric parts as in \eqref{ShallowWaterMomentum2D-1}. However, the corresponding vorticity does not satisfy a transport equation similar to \eqref{CurledTransport}. 

To be more precise,  let $ J_{ij}, I_{ij} \in \mathbb R^{n\times n} $, $ 1\leq i,j \leq n $, be matrices defined by
\begin{equation}
    \label{def:j-ij}
    (J_{ij})_{lk} := - \delta_{li}\delta_{kj} +  \delta_{ki}\delta_{lj},
    \qquad (I_{ij})_{lk} := \delta_{li}\delta_{kj} +  \delta_{ki}\delta_{lj}.
\end{equation}

\smallskip 

In analogy to \eqref{def:2d-non-sym}, we focus on the term 
\begin{equation}
    \label{def:n-d-non-symmetry}
    K:= \frac{m}{h^2} \cdot \nabla m + \frac{m}{h^2} \Div m =  \sum_{i=1}^n B_i \partial_i m,
\end{equation}
where 
\begin{equation}
    \label{def:B-i}
    (B_i)_{lk} = \frac{m_i}{h^2} \delta_{lk}  + 
        \frac{m_l}{h^2} \delta_{ki}.
\end{equation}
Then one can write
\begin{equation}
    \label{n-d-001}
    B_i = B_\mathrm{i,sym} + \sum_{j=1}^n A_{ij},
\end{equation}
where
\begin{equation}
    \label{n-d-002}
    \begin{aligned}
        B_\mathrm{i,sym}&:= \frac{m_i}{h^2} \mathbb I_n +
        \sum_{j=1}^n\frac{m_j}{2 h^2} I_{ij}, \\
        A_{ij} & :=
        \frac{m_j}{2 h^2} J_{ij}. 
    \end{aligned}
\end{equation}
Indeed,
\begin{align*}
    ( I_{ij} + J_{ij} )_{lk} & = 2 \delta_{ki} \delta_{lj}, \\
    \intertext{by definition, and therefore}
    (B_\mathrm{i,sym} + \sum_{j=1}^n A_{ij})_{lk} & = \frac{m_i}{h^2} \delta_{lk} + \sum_{j=1}^n \frac{m_j}{2h^2} (I_{ij} + J_{ij})_{lk} \\
    & = \frac{m_i}{h^2} \delta_{lk} + \sum_{j=1}^n \frac{m_j}{2h^2} 2 \delta_{ki}\delta_{lj} 
    = \frac{m_i}{h^2}\delta_{lk} + \frac{m_l}{h^2}\delta_{ki} \overset{\eqref{def:B-i}}{=} (B_i)_{lk}.
\end{align*}
Therefore, substituting \eqref{n-d-001} and \eqref{n-d-002} into \eqref{def:n-d-non-symmetry} yields 
\begin{equation}
    \label{n-d-003}
    \begin{aligned}
    K & = \sum_{i=1}^n B_\mathrm{i,sym} \partial_i m + \sum_{i,j=1}^n \frac{m_j}{2h^2} J_{ij} \partial_i m \\ & = \sum_{i=1}^n B_\mathrm{i,sym} \partial_i m + \sum_{i,j=1}^n \frac{m_j}{2h^2} J_{ij} \underbrace{(\partial_i m - \nabla m_i)}_{= (\nabla_\mathrm{skw} m)_i }  + \sum_{i,j=1}^n \frac{m_j}{2h^2} J_{ij} \nabla m_i,
    \end{aligned}
\end{equation}
where
\begin{equation}
    \label{n-d-005}
    J_{ij} \nabla \cdot m = - \partial_j m_i + \partial_i m_j = (\nabla_\mathrm{skw} m)_{ij},
\end{equation}
is the $ ij $-th component of the anti-symmetric gradient. 
In particular, \eqref{n-d-005} implies the duality
\begin{equation}
    \label{duality-n-d}
    \int J_{ij}\nabla f \cdot g \,d\vec x = - \int f (\nabla_\mathrm{skw} g)_{ij} \,d\vec x,
\end{equation}
in analogy to \eqref{duality-formula}.

Now we are ready to write the system \eqref{ShallowWaterMomentum} into symmetric form in higher dimensions. Let $ U := (h-1, m, \nabla h) = (h-1, m_1, m_2, \cdots, m_n, \partial_1 h, \partial_2 h, \cdots , \partial_n h) \in \mathbb R^{2n+1} $. Write, for $ i,j = 1,2,\cdots , n $,  
\begin{equation}
    \label{def:coefficient}
    \begin{aligned}
        M_0 &:= \diag (1, \frac{1}{h}\mathbb I_n, \mathbb I_n), &
        M_i &:= \begin{pmatrix}
            0 & \vec e_i^\top & 0\\ 
            \vec e_i & B_\mathrm{i,sym} & 0 \\
            0 & 0 & 0 \times \mathbb I_n
        \end{pmatrix}, \\
        M_{ij} &:= \begin{pmatrix}
            0 & 0 & 0 \\
            0 & 0 & M_{ij}'' \\
            0 & M_{ij}' & 0
        \end{pmatrix}, & & 
    \end{aligned}
\end{equation}
with 
\begin{equation}
    \label{def:m-ij'''}
    (M_{ij}')_{lk} := \delta_{li}\delta_{kj} \ \text{for}\ l, k = 1,2,\cdots n, \quad \text{and}  \quad M_{ij}'' = - (M_{ij}')^\top,
\end{equation}
and 
\begin{equation}
    \label{def:source}
    S := \begin{pmatrix}
    0 \\ \frac{(m\cdot\nabla)h}{h^3}  m \\ 0
    \end{pmatrix}, \qquad 
    S'':= \begin{pmatrix}
        0 \\ 
        - \sum_{i,j=1}^n \frac{m_j}{2h^2} J_{ij}  (\nabla_\mathrm{skw} m)_i - \sum_{i,j=1}^n \frac{m_j}{2h^2} J_{ij} \nabla m_i  \\
        0
    \end{pmatrix}.
\end{equation}
It is easy to verify that $ M_i $, $ i = 0,1,2,\cdots n $, are symmetric matrices and $ M_{ij} $, $ i,j = 1,2,\cdots n $ are constant anti-symmetric matrices. Then the system \eqref{ShallowWaterMomentum} can be written as
\begin{equation}
    \label{n-d-006}
    \underbrace{M_0 \partial_t U + \sum_{i=1}^n M_i \partial_i U + \sum_{i,j =1}^n M_{ij} \partial_{ij} U = S}_{\text{symmetric part}} \qquad  + \underbrace{S''}_{\mathclap{\text{non-symmetric part}}}. 
\end{equation}
While the symmetric part can be handled as in the one and two dimensional cases, the non-symmetric part requires the estimates of 
\begin{equation}
\label{id:m-2-u-vorticity}
    \nabla_\mathrm{skw} m = h {\nabla_\mathrm{skw} u} + u (\nabla h)^\top - \nabla h u^\top. 
\end{equation}
Meanwhile, from \eqref{MoMeNtUm}, one can calculate, for $ \forall \ l,k = 1,2,\cdots, n $, 
\begin{gather*}
    \partial_t (\partial_k u_l - \partial_l u_k) = - \partial_k (u\cdot \nabla u_l) + \partial_l (u\cdot \nabla u_k) \\ = - u \cdot \nabla (\partial_k u_l - \partial_l u_k) - \sum_{j=1}^n ( \partial_k u_j \partial_j u_l - \partial_l u_j \partial_j u_k) \\
    = - u \cdot \nabla (\partial_k u_l - \partial_l u_k) - \sum_{j=1}^n \lbrack  (\partial_k u_j - \partial_j u_k) \partial_j u_l + (\partial_j u_l - \partial_l u_j) \partial_j u_k \rbrack.
\end{gather*}
That is,
\begin{equation}
\label{CurledTransport-n-d-before}
    \partial_t \nabla_\mathrm{skw} u + u \cdot \nabla \nabla_\mathrm{skw} u= - \nabla_\mathrm{skw} u (\nabla u)^\top - \nabla u \nabla_\mathrm{skw} u.
\end{equation}
In particular, if $ \nabla_\mathrm{skw} u(t=0)=0 $, one can conclude from \eqref{CurledTransport-n-d-before} that, for all $ \forall \ t \geq 0 $,
\begin{equation}
    \label{irrotation}
    \nabla_\mathrm{skw} u(t) \equiv 0.
\end{equation}
Thus, from \eqref{id:m-2-u-vorticity} and \eqref{irrotation}, one can conclude that 
\begin{equation}
\label{est:skw-m}
    \nabla_\mathrm{skw}m = u (\nabla h)^\top - \nabla h u^\top \simeq U^2,
\end{equation}
where no loss of regularity appears. 
Thus, repeating the arguments as in Section \ref{sec:2d} leads to the following:
\begin{proposition}
    \label{prop:n-d}
    For $ n \geq 3 $, let 
    \begin{equation}
    \label{def:s-for-n-d}
        \frac{n}{2}+1 < s_n:= \begin{cases}
            \frac{n}{2}+ \frac{3}{2} & \text{if $ n $ is odd}, \\
            \frac{n}{2}+2 & \text{if $ n $ is even}.
        \end{cases}
    \end{equation}
    Consider initial data $ (h_0-1, u_0) \in H^{2 s_n+1}(\mathbb R) \times H^{2 s_n}(\mathbb R^n) $ for the system \eqref{ShallowWaterEquations} with $ \nabla_\mathrm{skw} u_0 = 0 $ and \eqref{h_initial-bound}. Then there exists some time $T\in (0,\infty)$ such that there exists a unique local-in-time solution $(h-1,m)\in L^{\infty}([0,T];H^{s_n+1}(\mathbb{R}))\times L^{\infty}([0,T];H^{s_n}(\mathbb{R}^n))$ to \eqref{ShallowWaterEquations}.
\end{proposition}

\section*{Acknowledgments}
Part of this work was compiled during the second author's visit to the Fields Institute during the program \emph{Thematic Program on Shocks and Singularities: Nonlinear Evolution Equations in Physical and Life Sciences}. BY graciously acknowledges the financial support by the Fields Institute, and thanks the organizers for their warm hospitality. 

\section*{Conflict of Interest}
On behalf of all authors, the corresponding author states that there is no conflict of interest.

\section*{Data Availability}
Data sharing is not applicable to this article as no new data were created or analyzed in this study.

\bibliographystyle{plain}
\bibliography{main}

\end{document}